\documentclass[final,3p,times]{elsarticle}
\biboptions{sort&compress}
\usepackage{cancel}
\usepackage{xcolor}
\usepackage{amsmath,amsfonts,amssymb,amscd,amsthm, mathrsfs}
\usepackage{subcaption}
\usepackage{extarrows}
\usepackage{lineno}
\usepackage{hyperref}					% Reference
\usepackage{verbatim}                   % 多行注释
\usepackage{booktabs}                   %表格
\usepackage{float}

\usepackage{tabularx}  % 用于调整表格宽度
\usepackage{booktabs}  % 美化表格线

\usepackage{graphicx}
\usepackage{epsfig}
\usepackage{epstopdf}
\usepackage[normalem]{ulem}

\hypersetup{colorlinks,%				% Reference color setup	
	linkcolor=blue,%
	citecolor=blue}

\newtheoremstyle{kai}
{3pt}{3pt}{}{}{\bfseries}{.}{.5em}{}
\makeatletter
\def\EquationsBySection{\def\theequation
	{\thesection.\arabic{equation}}%
	\@addtoreset{equation}{section}}

\newcommand\old[1]{}

\newtheorem{theorem}{Theorem}[section] % 如果不采用章节号做前缀， 则不用[section]
\newtheorem{lemma}[theorem]{Lemma}% 这句定义使得lem 环境和thm 共享编号
\newtheorem{assumption}[theorem]{Assumption}

\newtheorem{remark}[theorem]{Remark}
\newtheorem{example}{Example}[section]

\journal{Communications in Nonlinear Science and Numerical Simulation} 
\EquationsBySection
\makeatother

\begin{document}

\begin{frontmatter}
	\title{An explicit adaptive time-stepping scheme for superlinear stochastic diffusion systems
    }
    
    \author[aff1]{Xueqi Wen}
    \author[aff2]{Guozhen Li}
    \author[aff1]{Yuanping Cui}
    \author[aff1]{Xiaoyue Li\corref{cor1}}

    \address[aff1]{School of Mathematical Sciences,
    	Tiangong University, Tianjin 300387, China}
    
    \address[aff2]{School of Mathematics and Statistics,
    	Northeast Normal University, Changchun, Jilin 130024, China}
    
    \cortext[cor1]{Corresponding author.}
    
    \ead{lixy@tiangong.edu.cn}
    
    \fntext[funding]{  	
    	Research of Xiaoyue Li was supported by the National Natural Science Foundation
    	of China (No. 12371402) and the Tianjin Natural Science Foundation (24JCZDJC00830); Research of Yuanping Cui was supported by the National Natural Science Foundation of China (No. 12401216).}

    \begin{abstract} 	
  		This paper develops an adaptive time-stepping Euler--Maruyama (EM) scheme for stochastic diffusion systems with superlinearly growing coefficients. The adaptive timestep is chosen according to the superlinear growth of both drift and diffusion coefficients. To prevent excessively small timesteps, a truncated EM scheme is employed as a backstop whenever the adaptive timestep falls below a prescribed threshold.
  		By combining the stochastic analysis with the stopping time technique, we establish the strong convergence of the proposed method and obtain the optimal $1/2$-order strong convergence rate in the $L^q$-sense for $q>2$. {Finally, numerical experiments are carried out for stiff, nonstiff, and stochastic Lorenz systems  to validate the theoretical findings. The results indicate that the proposed scheme achieves superior accuracy and performance compared to various fixed-step and adaptive alternatives.}  
    
    \end{abstract}

	\begin{keyword}
		stochastic diffusion systems; adaptive time-stepping; superlinear growth; convergence rate.

	\end{keyword}
	
\end{frontmatter}

\section{Introduction} \label{sec1}

Stiff stochastic systems with superlinearly growing coefficients naturally arise in various applications, such as stochastic chemical kinetics \cite{rathinam2003stiffness, ilie2012adaptive} and stochastic reaction-diffusion systems \cite{ta2015an}. Because the solutions of these stiff systems involve components with widely differing time scales, explicit fixed-step methods typically suffer from severe stepsize restrictions during numerical simulations \cite[p. 298]{kloeden1992numerical}. Implicit methods, such as the semi-implicit EM scheme, are commonly employed to address this issue \cite{hu1996semi, mao2013strong}. However, their implementation can become computationally prohibitive for high-dimensional or complex systems, thereby limiting their practical applicability. Consequently, adaptive time-stepping methods offer a compelling alternative for the numerical approximation of stiff SDEs with superlinearly growing coefficients.

Adaptive time-stepping methods, designed to prevent numerical solutions from diverging, have been extensively studied over the past two decades \cite{chen2025strong}. One class of such methods for SDEs is based on local error control; see e.g., \cite{gaines1997variable, mauthner1998stepsize, burrage2002a, lamba2007an, ilie2015adaptive, shardlow2016on}. However, these methods often require interpolation of the Brownian path when a proposed timestep is rejected, which increases implementation complexity \cite{kelly2018adaptive, fang2019adaptive}. An alternative approach determines the timestep based on the current state, thereby avoiding step rejection altogether.

Kelly and Lord \cite{kelly2018adaptive} proposed an adaptive EM scheme that controls the potential unbounded growth of the superlinearly growing drift coefficient by adjusting the length of each timestep, with a backstop tamed scheme invoked when the timestep becomes too small. For SDEs with one-sided Lipschitz drift and linearly growing diffusion coefficients, they established a strong convergence rate of order $1/2$ in the $L^2$-sense. This adaptive strategy was later extended to semi-implicit adaptive numerical schemes \cite{kelly2022adaptive} for semi-linear SDEs with superlinearly growing drift and diffusion coefficients, where a strong convergence rate of order $(1-\varepsilon)/2$, $\varepsilon \in (0,1)$, was obtained. In particular, the $1/2$-order rate is recovered when the diffusion coefficient satisfies a linear growth condition (see \cite[Remark 6]{kelly2022adaptive}). Furthermore, Fang and Giles \cite{fang2020adaptive} proposed an adaptive time-stepping EM scheme for SDEs with one-sided Lipschitz drift and linearly growing diffusion, establishing a $1/2$-order strong convergence rate in the $L^q$ -sense $(q > 2)$ over finite time intervals.

Despite these advances, optimal $L^q$ $(q > 2)$ error estimates for adaptive time-stepping schemes remain underdeveloped when both the drift and diffusion coefficients are allowed to grow superlinearly. To bridge this gap, the purpose of this work is to develop an adaptive time-stepping scheme for SDEs with superlinearly growing drift and diffusion coefficients, and to establish the optimal $1/2$-order strong convergence rate for the proposed scheme.

Motivated by the adaptive time-stepping strategies proposed by Kelly and Lord \cite{kelly2018adaptive, kelly2022adaptive}, we develop an explicit adaptive time-stepping EM scheme with a backstop truncated EM method for SDEs where both drift and diffusion coefficients grow superlinearly. A key innovation of our approach lies in the design of the adaptive timestep. In \cite{kelly2022adaptive}, the authors proposed a semi-implicit adaptive scheme whose timestep depends only on the drift coefficient, achieving a $1/2 - \varepsilon$ ($\varepsilon \in (0, 1)$) strong convergence rate in the $L^2$-sense. In contrast, our adaptive timestep explicitly accounts for the superlinear growth of both the drift and diffusion coefficients. Furthermore, a backstop truncated EM scheme is seamlessly invoked when the adaptive timestep falls below a prescribed threshold to maintain numerical stability.

From an analytical perspective, establishing moment bounds in the adaptive setting presents a significant challenge. We emphasize that, unlike in the fixed-step framework, the moment boundedness of the continuous-time approximation does not generally imply the moment boundedness of the corresponding discrete-time approximation in the adaptive setting (see Remark \ref{rem-diff}). To overcome this difficulty, we employ Taylor expansions and rigorous discrete-time analysis to derive moment bounds for both the discrete-time numerical solution and its continuous-time extension. Utilizing stopping time arguments, we establish the strong convergence of the numerical scheme in the $L^q$
 -sense ($q > 2$) under a local Lipschitz assumption. Moreover, under an additional polynomial growth condition, we achieve the optimal $1/2$-order strong convergence rate in the $L^q$
 -sense.

Finally, we conduct a comprehensive comparison of the proposed scheme with existing methods, including the tamed EM, truncated EM, backward EM, and other adaptive scheme. Numerical experiments on stiff and nonstiff SDEs, as well as a stochastic Lorenz system, demonstrate that our proposed scheme achieves superior accuracy and better overall performance than the competing methods.
 
The rest of this paper is organized as follows. Section \ref{sec2} gives some preliminary results for the exact solution of SDE \eqref{equation}, and constructs the adaptive time-stepping scheme. Section \ref{sec3} proves the $p$th moment boundedness and establishes the $L^q$-strong convergence over a finite time interval. Section \ref{sec4} obtains the convergence rate. In Section \ref{sec5}, we show that the tamed EM scheme can be employed as the backstop scheme.
Section \ref{sec6} provides several numerical examples comparing the proposed scheme with several schemes to illustrate its numerical performance. Our conclusions and a discussion of possible future work are presented in Section \ref{sec7}.

\section{Preliminaries} \label{sec2}
Let $(\Omega,\mathcal{F},\mathbb{P})$ be a complete probability space with a filtration ${\{ {\mathcal{F}_t}\} _{t \ge 0}}$ satisfying the usual conditions (i.e., it is increasing and right continuous while $\mathcal{F}_0$ contains all $\mathbb{P}$-null sets). If $A$ is a vector or matrix, its transpose is denoted by $A^{\rm T}$. Let $W(t) = (W_1(t), W_2(t),...,W_m(t))^{\rm T} $ be an $m$-dimensional Brownian motion. Let $\left|  \cdot  \right|$ denote both the Euclidean norm in $\mathbb{R}^d$ and the Frobenius norm in $\mathbb{R}^{m \times d}$. Let $\langle { \cdot,  \cdot } \rangle $ denote the inner product of vectors in $\mathbb{R}^d$. For a set $G$, its indicator function is denoted by $\mathcal{I}_G$, namely $\mathcal{I}_G(a) = 1$ if $a \in G$ and $0$ otherwise. 
Denote ${\mathbb{R}_ + } = (0,\infty )$. $\mathbb{Q}$ denotes the set of all rational numbers.

Throughout this paper, we use $C$ to denote a generic positive constant whose value may change from place to place. Furthermore, we use $C_a$ to emphasize the dependence of the generic constant on parameter $a$. We consider a $d$-dimensional stochastic diffusion system described by the stochastic differential equation (SDE) 
\begin{equation}  \label{equation}
	\left\{  
	\begin{array}{l}
		dX(t) = f(X(t))dt + g(X(t))dW(t), \quad t \ge 0, \\
		X(0) = x_0,
	\end{array}
	\right.
\end{equation}
where $W(t)$ is an $m$-dimensional Brownian motion, and the drift and diffusion coefficients $f:{\mathbb{R}^d} \to {\mathbb{R}^d}$, $g:{\mathbb{R}^d} \to {\mathbb{R}^{d \times m}}$. For convenience we impose the following assumptions.

\begin{assumption}  \label{A2.1}
	For any $R > 0$, there exists a positive constant $C_R$ such that
	\begin{equation} \label{llc}
		\left| {f(x) - f(y)} \right| + \left| {g(x) - g(y)} \right| \le {C_R}\left| {x - y} \right|,
	\end{equation}
	for any $x, y \in \mathbb{R}^d$ with $\left| x \right| \vee \left| y \right| \le R$. Furthermore, there exists a pair of constants $p \ge 2$ and $\alpha \ge 0$ such that for any  $x \in \mathbb{R}^d$,
	\begin{equation}  \label{A2.1-1}
		\left\langle {x,f(x)} \right\rangle  + \frac{p-1}{2}{\left| {g(x)} \right|^2} \le \alpha (1 + {\left| x \right|^2}) .
	\end{equation}
\end{assumption}
\begin{remark}  \label{remark1}
	One deduces from \eqref{llc} that for any $R>0$ and $x\in \mathbb{R}^{d}$ with $|x|\leq R$, there exists a constant $C_{R}>0$ such that
	\begin{equation*}
		\begin{split}
			\left| {f(x)} \right| =& \left| {f(x) - f(0) + f(0)} \right| \le \left| {f(x) - f(0)} \right| + \left| {f(0)} \right|  
			\le  {C_R}\left| x \right| + \left| {f(0)} \right|   \le {C_{R}}\left( {1 + \left| x \right|} \right) . \\	
		\end{split}
	\end{equation*}
Similarly, we get ${\left| {g(x)} \right|^2}
		\le {C_{R}}{\left( {1 + \left| x \right|} \right)^2} $.
\end{remark}

Referring to \cite[Theorem 2.3.6]{mao2008stochastic}, SDE \eqref{equation} admits a unique global solution $X(t)$ on $[0,\infty)$ under Assumption \ref{A2.1}. Moreover, the solution $X(t)$  is bounded on every finite time interval. We state these results as the following lemma.

\begin{lemma} (\cite{li2019explicit, mao2008stochastic})  \label{EXt}
Under Assumption \ref{A2.1}, the SDE \eqref{equation} has a unique solution $X(t)$ on $[0,\infty)$ and satisfies
\begin{equation*}
	\mathop {\sup }\limits_{0 \le t \le T} \mathbb{E}{\left| {X(t)} \right|^p} \le C, \quad \forall T > 0.
\end{equation*}
Moreover, for any constant $R > |x_0|$, define
\begin{equation}  \label{sigmaR}
	{\sigma _R} = \inf \{ t \ge 0:\left| {X(t)} \right| \ge R\} .
\end{equation}
Then for any $T > 0$, 
\begin{equation}  \label{PsigmaR}
	\mathbb{P}({\sigma _R} \le T)  \le \frac{{{C}}}{{{R^p}}},
\end{equation}
where $C$ is a generic positive constant depending on $T,p,\alpha$ and $x_0$, but independent of $R$. 
\end{lemma}

In the following, we now construct an adaptive time-stepping numerical scheme for \eqref{equation} under Assumption \ref{A2.1}. We first estimate the growth rate of $f ( \cdot )$ and $g ( \cdot )$. Based on Remark \ref{remark1}, one can choose a strictly increasing continuous function $\varphi :[0, \infty) \to [1, \infty)$ such that ${\lim _{r \to  + \infty }}\varphi (r) =  + \infty $ and
\begin{equation}  \label{phi}
	\mathop {\sup }\limits_{\left| x \right| \le r} \left( {\frac{{\left| {f(x)} \right|}}{{1 + \left| x \right|}} \vee \frac{{{{\left| {g(x)} \right|}^2}}}{{{{(1 + \left| x \right|)}^2}}}} \right) \le \varphi (r),  \quad \forall r > 0.
\end{equation}
Denote by ${\varphi ^{ - 1}}( \cdot )$ the inverse function of $\varphi ( \cdot )$. It is clear that ${\varphi ^{ - 1}}:[\varphi (1),\infty ) \to {\mathbb{R}_ + }$  is strictly increasing and continuous. Choose a pair of constants $\hat{\delta}, \check{\delta}$ such that $0 <\hat{\delta} \le \check{\delta }\leq 1$. Define an  adaptive time-stepping function $\delta :{\mathbb{R}^d} \to (0,1]$ by
\begin{equation}  \label{timestep}
	\delta (x) = \hat \delta  \vee \left(\frac{\check{\delta} }{{\varphi \left( {\left| x \right|} \right)}}  \right) .
\end{equation}
Since $\varphi (|x|) \ge 1$ for any $x \in \mathbb{R}^d$, we get  $\hat{\delta}\leq \delta(x)\leq \check{\delta}$ for any $x\in \mathbb{R}^d$. Now, the adaptive EM scheme (ATS) is defined by
\begin{equation}  \label{scheme}
	\left\{  
	\begin{array}{l}
		Y_0 = x_0, \quad t_0 = 0, \\
		\delta_n := \delta(Y_{t_n}), \quad t_{n+1} = t_n + \delta_n, \quad n = 0,1,...,  \\
		{Y_{{t_{n + 1}}}} = {Y_{{t_n}}} + \left\{ {f({Y_{{t_n}}}){\delta _n} + g({Y_{{t_n}}})\Delta {W_n}} \right\}{\mathcal{I}_{\{ \hat \delta  < {\delta _n} \le \check{\delta} \} }}   + U(Y_{{t_n}}, \delta _n){\mathcal{I}_{\{ {\delta _n} = \hat \delta \} }}, \quad n = 0,1,...,
	\end{array}
	\right.
\end{equation}
where $\Delta {W_n} = W({t_{n + 1}}) - W({t_n})$.
Here $U:{\mathbb{R}^d} \times (0,1] \times \mathbb{R}^m \to {\mathbb{R}^d}$ denotes a measurable mapping, referred to as the backstop scheme, which may choose as some existing numerical schemes such as the truncated EM scheme (TEM), the tamed EM scheme (TaEM) and others. If we choose TEM as the backstop scheme, that is,
\begin{equation}\label{eq-tem} 
 	U(Y_{{t_n}}, \delta _n, \Delta{W_n})=  {f_{\hat \delta }({Y_{{t_n}}}){\delta _n} + g_{\hat \delta }({Y_{{t_n}}})\Delta {W_n}}  , \quad {f_{\hat \delta }}(x): = f({\pi _{\hat \delta }}(x)), \quad {g_{\hat \delta }}(x): = g({\pi _{\hat \delta }}(x)), 
\end{equation} 
where the truncated mapping ${\pi _{{\hat \delta }} }:{\mathbb{R}^d} \to {\mathbb{R}^d}$ is defined by
\begin{equation}  \label{pi}
	\begin{split}
		{\pi _{\hat \delta } }(x) = \left( {\left| x \right| \wedge {\varphi ^{ - 1}}\left( {K{\hat \delta }^{ - \gamma}} \right)} \right)\frac{x}{{\left| x \right|}} , \quad \gamma \in (0, 1/3], \quad K \ge \varphi(1),
	\end{split}
\end{equation}
we use the convention $x/|x| = \textbf{0}$ when $x = \textbf{0} \in \mathbb{R}^d$. 
The scheme \eqref{scheme}-\eqref{eq-tem} is called the adaptive EM scheme with backstop truncated EM scheme (ATS-TEM).  As $\hat{\delta} = \check{\delta}$, the scheme \eqref{scheme}-\eqref{eq-tem} degenerates to the TEM scheme with the fixed step. We refer the reader to \cite{yang2022convergence} for more details.

For convenience, define ${n_t} = \max \{ n:{t_n } \le t\}$ for any $t\geq 0$. Based on numerical scheme \eqref{scheme}-\eqref{eq-tem},  we define two versions of the continuous-time numerical solutions by
\begin{equation} \label{continuous}
	\bar Y(t) = Y_{t_{n_t}}, \quad  \forall t \ge 0,
\end{equation}
and 
\begin{equation}  \label{continuous1}
	{Y}(t)= {x_0} + \int_0^t {F(\bar Y(s))ds}  + \int_0^t {G(\bar Y(s))dW(s)} , \quad \forall t \ge 0,
\end{equation}
where 
\begin{equation}  \label{FG}
	\begin{split}
		F(x) =& f(x)\mathcal{I}_{\{ \hat{\delta}  < \delta (x) \le \check{\delta} \}}  + {f_{\hat{\delta}} }(x)\mathcal{I}_{\{ \delta (x) = \hat{\delta} \}} ,~~~~
		G(x) = g(x)\mathcal{I}_{\{ \hat{\delta}  < \delta (x) \le \check{\delta} \}}  + {g_{\hat{\delta}} }(x)\mathcal{I}_{\{ \delta (x) = \hat{\delta} \}}. 
	\end{split}
\end{equation}
Clearly, we observe that $Y({t_{{n_t}}}) = \bar Y({t_{{n_t}}}) = {Y_{{t_{{n_t}}}}}$ for any $t \ge 0$. 

In the following, we see that $F(\cdot)$ and $G(\cdot)$ have some nice properties. On the one hand, for every $x \in \mathbb{R}^d$ satisfying $\hat{\delta}<\delta(x) \le \check{\delta}$, it follows from \eqref{timestep} that 
 $$\varphi (|x|) =\frac{\check{\delta}}{\delta(x)}\leq \frac{\check{\delta}}{\hat{\delta}}:=\rho.$$
Furthermore, according to \eqref{phi} and \eqref{timestep} we have
\begin{equation}  \label{fg}
	\begin{split}
	\left| {f(x)} \right|\mathcal{I}_{\{\hat{\delta}<\delta(x)\leq \check{\delta}\}}\le& \varphi \left( {\left| x \right|} \right)\left( {1 + \left| x \right|} \right)\mathcal{I}_{\{\hat{\delta}<\delta(x)\leq \check{\delta}\}} \le \rho \left( {1 + \left| x \right|} \right). \\
	\end{split}
\end{equation}
Similarly, one can also deduce that
\begin{align}\label{eq2.14}
\left| {g(x)} \right|\mathcal{I}_{\{\hat{\delta}<\delta(x)\leq \check{\delta}\}} \le& {\varphi ^{\frac{1}{2}}}\left( {\left| x \right|} \right)\left( {1 + \left| x \right|} \right)\mathcal{I}_{\{\hat{\delta}<\delta(x)\leq \check{\delta}\}} \le {\rho ^{\frac{1}{2}}}\left( {1 + \left| x \right|} \right).
\end{align}
On the other hand, for every $x \in \mathbb{R}^d$ satisfying ${\delta (x)} = {\hat \delta}$, using \eqref{phi}  and the monotonicity $\varphi (\cdot)$, we obtain
\begin{equation}\label{fglinear}
\begin{aligned}  
		\left| {{f_{\hat \delta }}(x)} \right|\mathcal{I}_{\{{\delta (x)} = {\hat \delta}\}} =& \left| {f({\pi _{\hat \delta }}(x))} \right|\mathcal{I}_{\{{\delta (x)} 
= {\hat \delta}\}} \le \varphi \left( {\left| x \right| \wedge {\varphi ^{ - 1}}(K{{\hat \delta }^{ - \gamma}})} \right)\left( {1 + \left| x \right|} \right)\mathcal{I}_{\{{\delta (x)} = {\hat \delta}\}}  \le  K{{\hat \delta }^{ - \gamma}}\left( {1 + \left| x \right|} \right), \\
\end{aligned}
\end{equation}
and
\begin{equation}\label{eq2.16}
\begin{aligned}
		{\left| {{g_{\hat \delta }}(x)} \right|}\mathcal{I}_{\{{\delta (x)} = {\hat \delta}\}} =& {\left| {g({\pi _{\hat \delta }}(x))} \right|}\mathcal{I}_{\{{\delta (x)} = {\hat \delta}\}} \le \Big[\varphi \left( {\left| x \right| \wedge {\varphi ^{ - 1}}(K{{\hat \delta }^{ - \gamma}})} \right)\Big]^{\frac{1}{2}}{\left( {1 + \left| x \right|} \right)}\mathcal{I}_{\{{\delta (x)} = {\hat \delta}\}}
  \le K{{\hat \delta }^{ - \frac{\gamma}{2}}}{\left( {1 + \left| x \right|} \right)}.
	\end{aligned}
\end{equation}
Thus, combining \eqref{fg}-\eqref{eq2.16}, it follows from \eqref{FG} and the fact $\hat{\delta} \le 1$ that for any $x \in \mathbb{R}^d$, 
\begin{equation}  \label{Flinear}
	\begin{split}
		\left| {F(x)} \right| \le& \left| {f(x){{\cal I}_{\{ \hat \delta  < \delta (x) \le \check{\delta} \} }}} \right| + \left| {{f_{\hat \delta }}(x){{\cal I}_{\{ \delta (x) = \hat \delta \} }}} \right| \le (K + \rho ){{\hat \delta }^{ - \gamma }}\left( {1 + \left| x \right|} \right) , \\
		\left| {G(x)} \right| \le& \left| {g(x){{\cal I}_{\{ \hat \delta  < \delta (x) \le \check{\delta} \} }}} \right| + \left| {{g_{\hat \delta }}(x){{\cal I}_{\{ \delta (x) = \hat \delta \} }}} \right| \le (K + {\rho ^{\frac{1}{2}}}){{\hat \delta }^{ - \frac{\gamma }{2}}}\left( {1 + \left| x \right|} \right) .
	\end{split}
\end{equation}

The following lemma shows that $F( \cdot )$ and $G( \cdot )$ preserve the Khasminskii-type condition \eqref{A2.1-1}.
\begin{lemma}  \label{fgdelta}
Let Assumption \ref{A2.1} hold. Then, one has
\begin{equation} \label{fgdelta1}
	\left\langle {x,F(x)} \right\rangle  + \frac{{p - 1}}{2}{\left| {G(x)} \right|^2} \le {C_1}\left( {1 + {{\left| x \right|}^2}} \right), \quad  \forall x \in \mathbb{R}^d,
\end{equation}
where ${C_1} = \alpha ({2} + 1/{\varphi ^{ - 1}}(K))$.
\end{lemma} 
\begin{proof}
It follows from \eqref{FG} that
\begin{equation}  \label{kl8}
	\begin{split}
		&\left\langle {x,F(x)} \right\rangle  + \frac{{p - 1}}{2}{\left| {G(x)} \right|^2} \\
		=& \left\langle {x,f(x){\mathcal{I}_{\{ \hat{\delta}  < \delta (x) \le \check{\delta} \} }} + {f_{\hat \delta }}(x){\mathcal{I}_{\{ \delta (x) = \hat \delta \} }}} \right\rangle  + \frac{{p - 1}}{2}{\left| {g(x){\mathcal{I}_{\{ \hat{\delta}  < \delta (x) \le \check{\delta} \} }} + {g_{\hat \delta }}(x){\mathcal{I}_{\{ \delta (x) = \hat \delta \} }}} \right|^2} \\
		=& \left[ {\left\langle {x,f(x)} \right\rangle  + \frac{{p - 1}}{2}{{\left| {g(x)} \right|}^2}} \right]{\mathcal{I}_{\{ \hat{\delta}  < \delta (x) \le \check{\delta} \} }} + \left[ {\left\langle {x,{f_{\hat \delta }}(x)} \right\rangle  + \frac{{p - 1}}{2}{{\left| {{g_{\hat \delta }}(x)} \right|}^2}} \right]{{\cal I}_{\{ \delta (x) = \hat \delta \} }} \\
		&+ (p - 1)tr(g(x){g_{\hat \delta }}{(x)^{\rm T}}){\mathcal{I}_{\{ \hat{\delta}  < \delta (x) \le \check{\delta} \} }}{{\cal I}_{\{ \delta (x) = \hat \delta \} }} .
	\end{split}
\end{equation}
Note that
\begin{equation}  \label{eqq2.20}
	{\mathcal{I}_{\{ \hat{\delta}  < \delta (x) \le \check{\delta} \} }}{{\cal I}_{\{ \delta (x) = \hat \delta \} }} = 0.
\end{equation}
For any $x \in \mathbb{R}^d$ with $\left| x \right| \le {\varphi ^{ - 1}}(K{\hat{\delta} ^{ - \gamma}})$, from \eqref{pi} we have $f(x) = f({\pi _{\hat \delta }}(x)) = {f_{\hat \delta }}(x)$ and $g(x) = g({\pi _{\hat \delta }}(x)) = {g_{\hat \delta }}(x)$. Therefore, it follows from \eqref{A2.1-1} that
\begin{equation}\label{eq2.20}
\begin{aligned}
	\left[ {\left\langle {x,{f_{\hat \delta }}(x)} \right\rangle  + \frac{{p - 1}}{2}{{\left| {{g_{\hat \delta }}(x)} \right|}^2}} \right]{\mathcal{I}_{\{ \left| x \right| \le {\varphi ^{ - 1}}(K{{\hat \delta }^{ - \gamma }})\} }} 
	 \leq  \alpha (1 + {\left| x \right|^2}).
\end{aligned}
\end{equation}
For any $x \in \mathbb{R}^d$ with $\left| x \right| > {\varphi ^{ - 1}}(K{{\hat{\delta}} ^{ - \gamma}})$, using \cite[Lemma 2.4]{mao2015the} yields
\begin{equation}  \label{fgdelta2}
\begin{split}
	&\Big[\left\langle {x,{f_{\hat{\delta}} }(x)} \right\rangle  + \frac{p - 1}{2}{\left| {{g_{\hat{\delta}} }(x)} \right|^2}\Big]{\mathcal{I}_{\{ \left| x \right| > {\varphi ^{ - 1}}(K{{\hat \delta }^{ - \gamma }})\} }} \\
	\le& \frac{{\alpha \left| x \right|}}{{{\varphi ^{ - 1}}(K{{\hat{\delta}} ^{ - \gamma}})}}\left[ {1 + {{\left( {{\varphi ^{ - 1}}(K{{\hat{\delta}} ^{ - \gamma}})} \right)}^2}} \right] 
	\le \left[ {\frac{{\alpha \left| x \right|}}{{{\varphi ^{ - 1}}(K{{\hat \delta }^{ - \gamma }})}} + \alpha {{\left| x \right|}^2}} \right].
\end{split}
\end{equation}
Combining \eqref{eq2.20} and \eqref{fgdelta2}, using the facts ${{\varphi ^{ - 1}}}(\cdot)$ is a strictly increasing function and $0 < \hat{\delta} \le 1$, one sees
\begin{equation}\label{eqqq2.24}
\begin{aligned}
    &{\left\langle {x,{f_{\hat \delta }}(x)} \right\rangle  + \frac{{p - 1}}{2}{{\left| {{g_{\hat \delta }}(x)} \right|}^2}} \leq {\alpha (1 + {{\left| x \right|}^2}) + \frac{{\alpha \left| x \right|}}{{{\varphi ^{ - 1}}(K{{\hat \delta }^{ - \gamma }})}} + \alpha {{\left| x \right|}^2}} \\
    \le& \alpha \left( {{2} + \frac{1}{{{\varphi ^{ - 1}}(K{\hat{\delta} ^{ - \gamma }})}}} \right)(1 + {\left| x \right|^2}) 
    \le \alpha \left( {{2} + \frac{1}{{{\varphi ^{ - 1}}(K)}}} \right)(1 + {\left| x \right|^2}).
\end{aligned}
\end{equation}
Furthermore, inserting \eqref{eqq2.20} and \eqref{eqqq2.24} into \eqref{kl8}, and using \eqref{A2.1-1} yields
\begin{equation*}
\begin{split}
	&\left\langle {x,F(x)} \right\rangle  + \frac{{p - 1}}{2}{\left| {G(x)} \right|^2}\\
	\le& \left[ {\left\langle {x,f(x)} \right\rangle  + \frac{{p - 1}}{2}{{\left| {g(x)} \right|}^2}} \right]{{\cal I}_{\{ \hat \delta  < \delta (x) \le \check{\delta} \} }} + \alpha \left( {{2} + \frac{1}{{{\varphi ^{ - 1}}(K)}}} \right)(1 + {\left| x \right|^2}){{\cal I}_{\{ \delta (x) = \hat \delta \} }} \\
	\le& \alpha (1 + {\left| x \right|^2}){{\cal I}_{\{ \hat \delta  < \delta (x) \le \check{\delta} \} }} + \alpha \left( {{2} + \frac{1}{{{\varphi ^{ - 1}}(K)}}} \right)(1 + {\left| x \right|^2}){{\cal I}_{\{ \delta (x) = \hat \delta \} }} \\
	\le& \alpha \left( {{2} + \frac{1}{{{\varphi ^{ - 1}}(K)}}} \right)(1 + {\left| x \right|^2}) .
\end{split}
\end{equation*} 
The proof is therefore complete.
\end{proof}

\begin{remark}
For any $T>0$, a positive lower bound on the timestep $\hat{\delta} > 0$ guarantees that the time horizon $T$ is almost surely attainable, i.e., $\mathbb{P}(\omega: \exists M(\omega ) < \infty \; \text{s.t.} \; {t_{M(\omega )}} \ge T) = 1$ (see \cite[Theorem 1]{fang2020adaptive}), since for almost all $\omega \in \Omega$, there exists a positive integer $M(\omega) < \infty$ such that
\begin{align*}
	t_{M(\omega)} =	\sum_{n=0}^{M(\omega)-1}\delta_{n}\geq M(\omega) \hat{\delta} \geq T.
\end{align*}
\end{remark}

\begin{remark}   \label{rem-2}
Note that the stochastic time $t_n$ defined in \eqref{scheme} is a $\{ {{\cal F}_t}\} $-stopping time, i.e., $\{ {t_n} \le t\}  \in {\mathcal{F}_t}$ (see e.g., \cite{liu2017almost}). We then define a filtration at stopping time $t_n$ by
\begin{equation*} 
	{\mathcal{F}_{t_n} } = \{ A \in \mathcal{F}:A \cap \{\omega : {t_n(\omega)}  \le t\}  \in {\mathcal{F}_t} \; \text{for all} \; t \ge 0\} .
\end{equation*}
Hence, $\delta_n$ is $\mathcal{F}_{t_n}$-measurable, and $t_{n+1}$ is a $\mathcal{F}_{t_n}$-measurable random variable. 
\end{remark}

\begin{remark}   \label{remark-1}
Notice that each Brownian increment $W({t_{n + 1}}) - W({t_n})$ depends on the adaptive timestep $\delta_n = {t_{n + 1}} - {t_n}$, and from \eqref{scheme} we see ${\delta _n} = {\delta }({Y_{t_n}})$ is determined by $Y_{t_n}$. Hence, in general, $W({t_{n + 1}}) - W({t_n})$ is not independent of $\mathcal{F}_{t_n}$. Nevertheless, $W(t_{n+1}) - W(t_n)$ is $\mathcal{F}_{t_n}$-conditionally normally distributed (see \cite[Remark 2.2]{kelly2018adaptive}). Since $\delta_n$ is an $ {{\cal F}_{{t_n}}} $-measurable random variable and $W(t)$ is an $m$-dimensional Brownian motion, $W(t)$ and $W^2(t) - mt$ are two classical ${\{ {\mathcal{F}_t}\} _{t \ge 0}}$-adapted martingales. By applying Doob's martingale stopping theorem (see e.g., \cite[Theorem 1.3.3]{mao2008stochastic}),
		\begin{equation*} 
			\begin{split}
				\mathbb{E}(W({t_{n + 1}}) - W({t_n})|{\mathcal{F}_{{t_n}}}) = \mathbb{E}(W({t_{n + 1}})|{\mathcal{F}_{{t_n}}}) - W({t_n}) = \mathbf{0} \quad a.s.
			\end{split}
		\end{equation*}
		and
		 \begin{equation*}
		 	\mathbb{E}\left( {{W^2}({t_{n + 1}}) - m{t_{n + 1}}|{{\cal F}_{{t_n}}}} \right) = {W^2}({t_n}) - m{t_n} \quad a.s.
		 \end{equation*}
		which implies
		\begin{equation*}
			\begin{split}
				&\mathbb{E}(|W({t_{n + 1}}) - W({t_n}){|^2} - m{\delta _n}|{\mathcal{F}_{{t_n}}}) \\
				=& \mathbb{E}({W^2}({t_{n + 1}}) - 2\left\langle {W({t_{n + 1}}),W({t_n})} \right\rangle  + {W^2}({t_n}) - m{t_{n + 1}} + m{t_n}|{{\cal F}_{{t_n}}})\\
				=& \mathbb{E}({W^2}({t_{n + 1}}) - m{t_{n + 1}}|{{\cal F}_{{t_n}}}) - 2\mathbb{E}(\left\langle {W({t_{n + 1}}),W({t_n})} \right\rangle |{\mathcal{F}_{{t_n}}}) + \mathbb{E}({W^2}({t_n}) - m{t_n}|{{\cal F}_{{t_n}}}) + 2m{t_n}\\
				=& \mathbf{0} \quad a.s.
			\end{split}
		\end{equation*}
		The Brownian increment $W(t_{n+1}) -W(t_n)$ has the same distribution as the random variable $\sqrt{\delta_n} \chi$, $\chi  \sim N(0,{I_{m \times m}})$, where $I_{m \times m}$ is the identity matrix (see \cite[Lemma 4.5]{botija2024explicit}). Therefore, for any $i \ge 1$,  $A \in \mathbb{R}^{ d}$ and $B \in \mathbb{R}^{d \times m}$, we have
		\begin{equation} \label{kl7}
			\begin{split}
				&\mathbb{E}\left[ {{{\left( {\left\langle {A,B\Delta {W_n}} \right\rangle } \right)}^{2i - 1}}|{{\cal F}_{{t_n}}}} \right] = 0 \quad a.s., \\
				&\mathbb{E}\left[ {{{\left( {\left\langle {A,B\Delta {W_n}} \right\rangle } \right)}^{2i}}|{{\cal F}_{{t_n}}}} \right] \le {\left| A \right|^{2i}}{\left| B \right|^{2i}}\mathbb{E}\left[ {{{\left| {\Delta {W_n}} \right|}^{2i}}|{{\cal F}_{{t_n}}}} \right] \\
				&= {\left| A \right|^{2i}}{\left| B \right|^{2i}}\mathbb{E}\left( {{{\left| {\sqrt {{\delta _n}} \chi } \right|}^{2i}}|{\mathcal{F}_{{t_n}}}} \right) = {\left| A \right|^{2i}}{\left| B \right|^{2i}}\delta _n^i\mathbb{E}\left( {{{\left| \chi  \right|}^{2i}}} \right) \le {\left| A \right|^{2i}}{\left| B \right|^{2i}}(2i - 1)!!{m^i}\delta _n^i \quad a.s.
			\end{split}
		\end{equation}	
	\end{remark}

\section{Convergence}  \label{sec3}

To establish the strong convergence of the ATS-TEM \eqref{scheme}-\eqref{eq-tem}, we first give the $p$th moment boundedness of the numerical solutions.

\subsection{Moment boundedness of the numerical solutions}
\begin{theorem}  \label{th2.1}
	Let Assumption \ref{A2.1} hold. For any $T>0$, there exists a positive constant $C$ such that 
	\begin{equation}  \label{kl5}
		\mathop {\sup }\limits_{0 \le t \le T} \mathbb{E}\left| {\bar Y(t)} \right|^p \le C  \quad  \text{and}  \quad  \mathop {\sup }\limits_{0 \le t \le T} \mathbb{E}\left| {Y(t)} \right|^p \le C ,
	\end{equation}
where $p$ is given in Assumption \ref{A2.1}.
\end{theorem}
\begin{proof}
For any integer $n \ge 0$, from \eqref{scheme}, \eqref{eq-tem} and \eqref{FG} we have
\begin{equation*}
\begin{split}
	{\left| {{Y_{{t_{n + 1}}}}} \right|^2} =& {\left| {{Y_{{t_n}}} + F({Y_{{t_n}}}){\delta _n} + G({Y_{{t_n}}})\Delta {W_n}} \right|^2} \\
	=& {\left| {{Y_{{t_n}}}} \right|^2} + 2\left\langle {{Y_{{t_n}}},F({Y_{{t_n}}})} \right\rangle {\delta _n} + {\left| {F({Y_{{t_n}}})} \right|^2}{\left| {{\delta _n}} \right|^2} + {\left| {G({Y_{{t_n}}})\Delta {W_n}} \right|^2} \\
	&+ 2\left\langle {{Y_{{t_n}}} + F({Y_{{t_n}}}){\delta _n},G({Y_{{t_n}}})\Delta {W_n}} \right\rangle .
\end{split}
\end{equation*}
We therefore get
\begin{equation*}
\begin{split}
	{\left( {1 + {{\left| {{Y_{{t_{n + 1}}}}} \right|}^2}} \right)^{\frac{p}{2}}} = {\left( {1 + {{\left| {{Y_{{t_n}}}} \right|}^2}} \right)^{\frac{p}{2}}}{\left( {1 + {\xi _n}} \right)^{\frac{p}{2}}},
\end{split}
\end{equation*}
where 
\begin{equation}  \label{xii}
	{\xi _n} = \frac{{2\left\langle {{Y_{{t_n}}},F({Y_{{t_n}}})} \right\rangle {\delta _n} + {{\left| {F({Y_{{t_n}}})} \right|}^2}\delta _n^2 + {{\left| {G({Y_{{t_n}}})\Delta {W_n}} \right|}^2} + 2\left\langle {{Y_{{t_n}}} + F({Y_{{t_n}}}){\delta _n},G({Y_{{t_n}}})\Delta {W_n}} \right\rangle }}{{1 + {{\left| {{Y_{{t_n}}}} \right|}^2}}} .
\end{equation}
For the given constant $p \ge 2$, choose a non-negative integer $k$ such that $2k < p \le 2(k+1)$. Using Remark \ref{rem-2} and \cite[Lemma 3.3]{yang2018explicit} one gets
\begin{equation}  \label{sum1}
	\begin{split}
		&\mathbb{E}\left[ {{{\left( {1 + {{\left| {{Y_{{t_{n + 1}}}}} \right|}^2}} \right)}^{\frac{p}{2}}}|{\mathcal{F}_{{t_n}}}} \right] \\
		\le& {\left( {1 + {{\left| {{Y_{{t_n}}}} \right|}^2}} \right)^{\frac{p}{2}}}\left[ {1 + \frac{p}{2}\mathbb{E}\left( {{\xi _n}|{\mathcal{F}_{{t_n}}}} \right) + \frac{{p(p - 2)}}{8}\mathbb{E}\left( {\xi _n^2|{\mathcal{F}_{{t_n}}}} \right) + \mathbb{E}\left( {\xi _n^3{P_k}({\xi _n})|{\mathcal{F}_{{t_n}}}} \right)} \right] \quad a.s.
	\end{split}
\end{equation}
where $P_k(\cdot)$ represents a $k$th-order polynomial whose coefficients depend only on $p$. Using \eqref{kl7}, \eqref{Flinear}, Remark \ref{remark-1}, the facts $ \delta_n \le \check{\delta}$ and $\rho = \check{\delta}/\hat{\delta}$ we have
\begin{equation}  \label{xi1}
	\begin{split}  
		&\mathbb{E}\left( {\xi _n|{{\cal F}_{{t_n}}}} \right) \\
		=& {\left( {1 + {{\left| {{Y_{{t_n}}}} \right|}^2}} \right)^{ - 1}}\mathbb{E}\Big[ {2\left\langle {{Y_{{t_n}}},F({Y_{{t_n}}})} \right\rangle {\delta _n} + {{\left| {F({Y_{{t_n}}})} \right|}^2}\delta _n^2 + {{\left| {G({Y_{{t_n}}})\Delta {W_n}} \right|}^2}}  \\
		&  { + 2\left\langle {{Y_{{t_n}}} + F({Y_{{t_n}}}){\delta _n},G({Y_{{t_n}}})\Delta {W_n}} \right\rangle |{{\cal F}_{{t_n}}}} \Big] \\
		=& {\left( {1 + {{\left| {{Y_{{t_n}}}} \right|}^2}} \right)^{ - 1}}\left[ {2\left\langle {{Y_{{t_n}}},F({Y_{{t_n}}})} \right\rangle  + {{\left| {G({Y_{{t_n}}})} \right|}^2}} \right]{\delta _n} + {\left( {1 + {{\left| {{Y_{{t_n}}}} \right|}^2}} \right)^{ - 1}}{\left| {F({Y_{{t_n}}})} \right|^2}\delta _n^2 \\
		\le& {\left( {1 + {{\left| {{Y_{{t_n}}}} \right|}^2}} \right)^{ - 1}}\left[ {2\left\langle {{Y_{{t_n}}},F({Y_{{t_n}}})} \right\rangle  + {{\left| {G({Y_{{t_n}}})} \right|}^2}} \right]{\delta _n} + C{\check{\delta} ^{2(1 - \gamma )}} \quad a.s.
	\end{split}
\end{equation}
Similarly, using \eqref{Flinear}, \eqref{kl7} and Remark \ref{remark-1}, one computes
\begin{equation}   \label{xi2}
	\begin{split}
		\mathbb{E}\left( {\xi _n^2|{{\cal F}_{{t_n}}}} \right) \le& 4{\left( {1 + {{\left| {{Y_{{t_n}}}} \right|}^2}} \right)^{ - 2}}{\left| {Y_{{t_n}}^{\rm{T}}G({Y_{{t_n}}})} \right|^2}{\delta _n} + C{\check{\delta} ^{2(1-\gamma)}} \\
		\le& 4{\left( {1 + {{\left| {{Y_{{t_n}}}} \right|}^2}} \right)^{ - 1}}{\left| {G({Y_{{t_n}}})} \right|^2}{\delta _n} + C{\check{\delta} ^{2(1-\gamma)}} \quad a.s. 
	\end{split}
\end{equation}
and
\begin{equation*}
	\mathbb{E}\left( {\xi _n^3{P_k}({\xi _n})|{\mathcal{F}_{{t_n}}}} \right) \le C{\check{\delta} ^{2(1 - \gamma )}} \quad  a.s.
\end{equation*}
Substituting \eqref{xi1} and \eqref{xi2} into \eqref{sum1}, using Assumption \ref{A2.1} and $\gamma \in (0, 1/3]$ yields
\begin{equation}  \label{a4} 
\begin{split}
	&\mathbb{E}\left[ {{{\left( {1 + {{\left| {{Y_{{t_{n + 1}}}}} \right|}^2}} \right)}^{\frac{p}{2}}}|{\mathcal{F}_{{t_n}}}} \right] \\
	\le& \left( {1 + C{\check{\delta} ^{2(1-\gamma)}}} \right){\left( {1 + {{\left| {{Y_{{t_n}}}} \right|}^2}} \right)^{\frac{p}{2}}} + \frac{{p{\delta _n}}}{2}{\left( {1 + {{\left| {{Y_{{t_n}}}} \right|}^2}} \right)^{\frac{p}{2} - 1}}\left( {2\left\langle {{Y_{{t_n}}},F({Y_{{t_n}}})} \right\rangle  + (p - 1){{\left| {G({Y_{{t_n}}})} \right|}^2}} \right)  \\
	\le& (1 + C\check{\delta} ){\left( {1 + {{\left| {{Y_{{t_n}}}} \right|}^2}} \right)^{\frac{p}{2}}}  \quad a.s.
\end{split}
\end{equation}
For any $t \ge 0$, by the definition of $n_t$, \eqref{a4} then gives
	\begin{equation}  \label{mn1}
		\mathbb{E}\left[ {{{\left( {1 + {{\left| {{Y_{{t_{{n_t} + 2}}}}} \right|}^2}} \right)}^{\frac{p}{2}}}|{\mathcal{F}_{{t_{{n_t} + 1}}}}} \right] \le (1 + C\check{\delta} ){\left( {1 + {{\left| {{Y_{{t_{{n_t} + 1}}}}} \right|}^2}} \right)^{\frac{p}{2}}} \quad a.s.
	\end{equation}
Taking conditional expectations on both sides of the above inequality and using \eqref{a4} one gets
	\begin{equation*}  
		\begin{split}
			&\mathbb{E}\left[ {\mathbb{E}\left( {{{\left( {1 + {{\left| {{Y_{{t_{{n_t}+2}}}}} \right|}^2}} \right)}^{\frac{p}{2}}}|{{\cal F}_{{t_{{n_t} + 1}}}}} \right)|{{\cal F}_{{t_{{n_t}}}}}} \right] \\
			\le& (1 + C\check{\delta} )\mathbb{E}\left[ {{{\left( {1 + {{\left| {{Y_{{t_{{n_t} + 1}}}}} \right|}^2}} \right)}^{\frac{p}{2}}}|{{\cal F}_{{t_{{n_t}}}}}} \right] \le (1 + C\check{\delta} )^2{\left( {1 + {{\left| {{Y_{{t_{{n_t}}}}}} \right|}^2}} \right)^{\frac{p}{2}}}\quad a.s.
		\end{split}
	\end{equation*}
	Since ${\mathcal{F}_{{t_{n_t }}}} \subset {\mathcal{F}_{{t_{n_t + 1}}}}$, the above inequality implies that 
	\begin{equation*}
		\mathbb{E}\left[ {{{\left( {1 + {{\left| {{Y_{{t_{{n_t}+2}}}}} \right|}^2}} \right)}^{\frac{p}{2}}}|{{\cal F}_{{t_{{n_t}}}}}} \right] \le (1 + C\check{\delta} )^2{\left( {1 + {{\left| {{Y_{{t_{{n_t} }}}}} \right|}^2}} \right)^{\frac{p}{2}}} \quad a.s. 
	\end{equation*}
	Iterating the above argument and repeatedly applying \eqref{a4}, we obtain that 
	\begin{equation*}
		\mathbb{E}\left[ {{{\left( {1 + {{\left| {{Y_{{t_{{n_t}}}}}} \right|}^2}} \right)}^{\frac{p}{2}}}|{{\cal F}_0}} \right] \le {(1 + C\check{\delta} )^{{n_t}}}{\left( {1 + {{\left| {{x_0}} \right|}^2}} \right)^{\frac{p}{2}}}  \quad a.s.
\end{equation*}
From \eqref{timestep} one sees $\hat{\delta} \le \delta_n \le \check{\delta}$, which implies ${n_t}\hat \delta  \le t$ for any $t \in [0, T]$. In addition, using the facts
\begin{equation*}
	{(1 + x)^a} \le {e^{ax}}, \quad \forall a \ge 0, \; x \ge 0,
\end{equation*}
and $\check{\delta} = \rho \hat{\delta}$ we get
\begin{equation*}
\begin{split}
	\mathbb{E}\left[ {{{\left( {1 + {{\left| {{Y_{{t_{{n_t}}}}}} \right|}^2}} \right)}^{\frac{p}{2}}}|{{\cal F}_0}} \right] \le {e^{C{n_t}\check{\delta} }}{\left( {1 + {{\left| {{x_0}} \right|}^2}} \right)^{\frac{p}{2}}} \le {e^{C\rho {n_t}\hat \delta }}{\left( {1 + {{\left| {{x_0}} \right|}^2}} \right)^{\frac{p}{2}}} \le {e^{C\rho T}}{\left( {1 + {{\left| {{x_0}} \right|}^2}} \right)^{\frac{p}{2}}} \quad a.s.
\end{split}
\end{equation*}
Taking expectations on both sides yields
\begin{equation}  \label{mn5}
	\mathbb{E}\left[ {{{\left( {1 + {{\left| {{Y_{{t_{{n_t}}}}}} \right|}^2}} \right)}^{\frac{p}{2}}}} \right] \le C. 
\end{equation}
The first assertion is a direct consequence of \eqref{continuous} and \eqref{mn5}, i.e.,
\begin{equation}  \label{mn6}
		\mathop {\sup }\limits_{0 \le t \le T} \mathbb{E}{\left| {\bar Y(t)} \right|^p} = \mathop {\sup }\limits_{0 \le t \le T} \mathbb{E}{\left| {{Y_{{t_{{n_t}}}}}} \right|^p} \le \mathop {\sup }\limits_{0 \le t \le T} \mathbb{E}\left[ {{{\left( {1 + {{\left| {{Y_{{t_{{n_t}}}}}} \right|}^2}} \right)}^{\frac{p}{2}}}} \right] \le C.
\end{equation}
Furthermore, from \eqref{continuous} and \eqref{continuous1} we see
\begin{equation*}
	Y(t) = {Y_{{t_{{n_t}}}}} + F({Y_{{t_{{n_t}}}}})(t - {t_{{n_t}}}) + G({Y_{{t_{{n_t}}}}})(W(t) - W({t_{{n_t}}})),  \quad \forall t \in [0, T].
\end{equation*}
Applying Remarks \ref{rem-2} and \ref{remark-1} one gets
\begin{equation*}
\begin{split}
	&\mathbb{E}\left( {{{\left| {Y(t)} \right|}^p}|{\mathcal{F}_{{t_{{n_t}}}}}} \right) \\
	\le& {3^{p - 1}}{\left| {{Y_{{t_{{n_t}}}}}} \right|^p} + {3^{p - 1}}{\left| {F({Y_{{t_{{n_t}}}}})} \right|^p}{(t - {t_{{n_t}}})^p} + {3^{p - 1}}{\left| {G({Y_{{t_{{n_t}}}}})} \right|^p}\mathbb{E}\left( {{{\left| {W(t) - W({t_{{n_t}}})} \right|}^p}|{\mathcal{F}_{{t_{{n_t}}}}}} \right)  \\
	\le& {3^{p - 1}}{\left| {{Y_{{t_{{n_t}}}}}} \right|^p} + {3^{p - 1}}{(K + \rho )^p}{{\hat \delta }^{ - p\gamma }}\delta _{{n_t}}^p{\left( {1 + \left| {{Y_{{t_{{n_t}}}}}} \right|} \right)^p} + {3^{p - 1}}{(K + {\rho ^{\frac{1}{2}}})^p}{{\hat \delta }^{ - \frac{{p\gamma }}{2}}}\delta _{{n_t}}^{\frac{p}{2}}{\left( {1 + \left| {{Y_{{t_{{n_t}}}}}} \right|} \right)^p} \\
	\le&  C\left( {1 + {{\left| {{Y_{{t_{{n_t}}}}}} \right|}^p}} \right) \quad a.s.
\end{split}
\end{equation*}
Then, taking expectations on both sides and using \eqref{mn6} yields
\begin{equation*}  
	\mathbb{E}{\left| {Y(t)} \right|^p} \le C\left( {1 + \mathbb{E}{{\left| {{Y_{{t_{{n_t}}}}}} \right|}^p}} \right) \le C.
\end{equation*}
Therefore, we get the desired result that
\begin{equation*}
	\mathop {\sup }\limits_{0 \le t \le T} \mathbb{E}{\left| {Y(t)} \right|^p} \le C .
\end{equation*}
This completes the proof.
\end{proof}

\begin{remark}  \label{rem-diff}
In the boundedness analysis of fixed–step numerical schemes (see e.g., \cite{mao2015the, sabanis2016euler}), one often exploits an inequality of the form ${\sup _{0 \le t \le T}}\mathbb{E}{\left| {{Y_{{t_{{n_t}}}}}} \right|^p} \le {\sup _{0 \le t \le T}}\mathbb{E}{\left| {Y(t)} \right|^p} $. Such an inequality is valid when the underlying grid is deterministic. However, it may fail for adaptive schemes, since the time grid points $t_n$ are stopping times. Indeed, for a fixed $t \in [0,T]$, the index $n_t$ and the grid point $t_{n_t}$ are random and depend on the sample point $\omega$. As a result, $t_{n_t}$ may select, on each realization, different time points at which $\left| {{Y_ \cdot }} \right|$ takes large values. Consequently, it becomes necessary to establish the $p$th moment boundedness of $Y_{t_{n_t}}$ for any $t \in [0, T]$.
\end{remark}

\begin{lemma}  \label{lemma-stop}  
	Let Assumption \ref{A2.1} hold. For any constant $R > |x_0|$, define a stopping time by
\begin{equation}    \label{varthetaR}
		{\vartheta _{R, \hat{\delta}}} = \inf \{ t \ge 0:\left| {Y(t)} \right| \ge R\}.
\end{equation}
Then, for any $T > 0$,
	\begin{equation}  \label{Pvar}
		\mathbb{P}({\vartheta _{R, \hat{\delta}}} \le T)  \le \frac{{{C}}}{{{R^p}}}.
	\end{equation}
\end{lemma}
\begin{proof}
Fix $T > 0$, for any $t \in [0,T]$, applying It\^o's formula to \eqref{continuous1} and using Lemma \ref{fgdelta}, we obtain 
\begin{equation*}
\begin{split}
	{\left| {Y(t)} \right|^p} \le& {\left| {{x_0}} \right|^p} + p\int_0^t {{{\left| {Y(s)} \right|}^{p - 2}}\left( {\left\langle {Y(s),F(\bar Y(s))} \right\rangle  + \frac{{p - 1}}{2}{{\left| {G(\bar Y(s))} \right|}^2}} \right)ds}+M(t) \\
	=& {\left| {{x_0}} \right|^p} + p\int_0^t {{{\left| {Y(s)} \right|}^{p - 2}}\left( {\left\langle {\bar Y(s),F(\bar Y(s))} \right\rangle  + \frac{{p - 1}}{2}{{\left| {G(\bar Y(s))} \right|}^2}} \right)ds} \\
	&+ p\int_0^t {{{\left| {Y(s)} \right|}^{p - 2}}\left\langle {Y(s) - \bar Y(s),F(\bar Y(s))} \right\rangle ds} +M(t)\\
	\le& {\left| {{x_0}} \right|^p} + p{C_1}\int_0^t {{{\left| {Y(s)} \right|}^{p - 2}}\left( {1 + {{\left| {\bar Y(s)} \right|}^2}} \right)ds}  \\
	&+ p\int_0^t {{{\left| {Y(s)} \right|}^{p - 2}}\left| {Y(s) - \bar Y(s)} \right|\left| {F(\bar Y(s))} \right|ds}  + M(t),
\end{split}
\end{equation*}
where $M(t) = p\int_0^t {{{\left| {Y(s)} \right|}^{p - 2}}{Y^{\rm T}}(s)G(\bar Y(s))dW(s)} $ is a local martingale with initial value $0$. This implies
\begin{equation*}
\begin{split}
	\mathbb{E}{\left| {Y(t \wedge {\vartheta _{R, \hat{\delta}}})} \right|^p} \le& {\left| {{x_0}} \right|^p} + p{C_1}\mathbb{E}\left( {\int_0^{t \wedge {\vartheta _{R,\hat \delta }}} {{{\left| {Y(s)} \right|}^{p - 2}}ds} } \right) + p{C_1}\mathbb{E}\left( {\int_0^{t \wedge {\vartheta _{R,\hat \delta }}} {{{\left| {Y(s)} \right|}^{p - 2}}{{\left| {\bar Y(s)} \right|}^2}ds} } \right) \\
	&+ p\mathbb{E}\left( {\int_0^{t \wedge {\vartheta _{R,\hat \delta }}} {{{\left| {Y(s)} \right|}^{p - 2}}\left| {Y(s) - \bar Y(s)} \right|\left| {F(\bar Y(s))} \right|ds} } \right) .
\end{split}
\end{equation*}
Using Young's inequality and H\"older's inequality one has
\begin{equation}  \label{kl}
\begin{split}
	\mathbb{E}{\left| {Y(t \wedge {\vartheta _{R, \hat{\delta}}})} \right|^p} 
	\le&  {\left| {{x_0}} \right|^p} + 2{C_1}T + C\mathbb{E}\left( {\int_0^T {{{\left| {Y(t)} \right|}^p}dt} } \right) + C\mathbb{E}\left( {\int_0^T {{{\left| {\bar Y(t)} \right|}^p}dt} } \right)  \\
	&+ C\mathbb{E}\left( {\int_0^T {{{\left| {Y(t) - \bar Y(t)} \right|}^{\frac{p}{2}}}{{\left| {F(\bar Y(t))} \right|}^{\frac{p}{2}}}dt} } \right) \\
	\le& {\left| {{x_0}} \right|^p} + 2{C_1}T + C\int_0^T {\mathbb{E}{{\left| {Y(t)} \right|}^p}dt}  + C\int_0^t {\mathbb{E}{{\left| {\bar Y(t)} \right|}^p}dt} \\
	&+ C\int_0^T {{{\left( {\mathbb{E}{{\left| {Y(t) - \bar Y(t)} \right|}^p}} \right)}^{\frac{1}{2}}}{{\left( {\mathbb{E}{{\left| {F(\bar Y(t))} \right|}^p}} \right)}^{\frac{1}{2}}}dt} .
\end{split}
\end{equation}
For any $t \in [0, T]$, it follows from \eqref{continuous} and \eqref{continuous1} that
\begin{equation*}
\begin{split}
	Y(t) - \bar Y(t) = Y(t) - {Y_{{t_{{n_t}}}}} = F({Y_{{t_{{n_t}}}}})(t - {t_{{n_t}}}) + G({Y_{{t_{{n_t}}}}})(W(t) - W({t_{{n_t}}})).
\end{split}
\end{equation*}
Using \eqref{Flinear}, \eqref{kl7} and Remark \ref{remark-1} yields
\begin{equation}  \label{kl3}
\begin{split}
	&\mathbb{E}\left( {{{\left| {Y(t) - \bar Y(t)} \right|}^p}|{\mathcal{F}_{{t_{{n_t}}}}}} \right) \\
	\le& {2^{p - 1}}{\left| {F({Y_{{t_{{n_t}}}}})} \right|^p}\mathbb{E}\left( {{{\left| {t - t_{n_t}} \right|}^p}|{\mathcal{F}_{{t_{{n_t}}}}}} \right) + {2^{p - 1}}{\left| {G({Y_{{t_{{n_t}}}}})} \right|^p}\mathbb{E}\left( {{{\left| {W(t) - W({t_{{n_t}}})} \right|}^p}|{\mathcal{F}_{{t_{{n_t}}}}}} \right) \\
	\le& {2^{p - 1}}{(K + \rho )^p}{{\hat \delta }^{ - p\gamma }}{{\check{\delta} }^p}{\left( {1 + \left| {{Y_{{t_{{n_t}}}}}} \right|} \right)^p} + {2^{p - 1}}{(K + {\rho ^{\frac{1}{2}}})^p}{{\hat \delta }^{ - \frac{{p\gamma }}{2}}}{{\check{\delta} }^{\frac{p}{2}}}{\left( {1 + \left| {{Y_{{t_{{n_t}}}}}} \right|} \right)^p} \\
	\le& C{{\check{\delta} }^{\frac{p}{2}(1 - \gamma )}}\left( {1 + {{\left| {{Y_{{t_{{n_t}}}}}} \right|}^p}} \right) \quad a.s.
\end{split}
\end{equation}
Taking expectations on both sides one has 
\begin{equation}  \label{kl1}
\begin{split}
	\mathbb{E}{\left| {Y(t) - \bar Y(t)} \right|^p} \le C{{\check{\delta} }^{\frac{p}{2}(1 - \gamma )}}\left( {1 + \mathbb{E}{{\left| {\bar Y(t)} \right|}^p}} \right).
\end{split}
\end{equation}
Besides, using \eqref{Flinear} we have
\begin{equation}  \label{kl2}
\begin{split}
	\mathbb{E}{\left| {F(\bar Y(t))} \right|^p} \le {(K + \rho )^p}{{\hat \delta }^{ - p\gamma }}\mathbb{E}{\left( {1 + \left| {{Y_{{t_{{n_t}}}}}} \right|} \right)^p} \le C{{\hat \delta }^{ - p\gamma }}\left( {1 + \mathbb{E}{{\left| {\bar Y(t)} \right|}^p}} \right) .
\end{split}
\end{equation}
Due to the fact $\gamma \in (0, 1/3]$, it follows from \eqref{kl1} and \eqref{kl2} that
\begin{equation}  \label{kl10}
\begin{split}
	{\left( {\mathbb{E}{{\left| {Y(t) - \bar Y(t)} \right|}^p}} \right)^{\frac{1}{2}}}{\left( {\mathbb{E}{{\left| {F(\bar Y(t))} \right|}^p}} \right)^{\frac{1}{2}}} \le C\rho^{p\gamma}{\check{\delta}^{\frac{p}{4}(1 - 3\gamma )}}\left( {1 + \mathbb{E}{{\left| {\bar Y(t)} \right|}^p}} \right)
	 \le C\left( {1 + \mathbb{E}{{\left| {\bar Y(t)} \right|}^p}} \right).
\end{split}
\end{equation}
Substituting \eqref{kl10} into \eqref{kl}, using Theorem \ref{th2.1} yields
\begin{equation}  \label{kl14}
\begin{split}
	\mathbb{E}{\left| {Y(t \wedge {\vartheta _{R,\hat \delta }})} \right|^p} \le& C + C\left( {\int_0^T {\mathbb{E}{{\left| {Y(t)} \right|}^p}dt}  + \int_0^T {\mathbb{E}{{\left| {\bar Y(t)} \right|}^p}dt} } \right) \\
	\le&  C + C\left( {\int_0^T {\mathop {\sup }\limits_{0 \le u \le t} \mathbb{E}{{\left| {Y(u)} \right|}^p}dt}  + \int_0^T {\mathop {\sup }\limits_{0 \le u \le t} \mathbb{E}{{\left| {\bar Y(u)} \right|}^p}dt} } \right) \le C .
\end{split}
\end{equation}
Using \eqref{kl14} and \eqref{varthetaR} one sees
\begin{equation*}
\begin{split}
	{C} \ge& \mathbb{E}{\left| {Y(T \wedge {\vartheta _{R,\hat \delta }})} \right|^p} = \mathbb{E}\left( {{{\left| {Y(T \wedge {\vartheta _{R,\hat \delta }})} \right|}^p}{{\cal I}_{\{ {\vartheta _{R,\hat \delta }} \le T\} }}} \right) + \mathbb{E}\left( {{{\left| {Y(T \wedge {\vartheta _{R,\hat \delta }})} \right|}^p}{{\cal I}_{\{ {\vartheta _{R,\hat \delta }} > T\} }}} \right) \\
	\ge& \mathbb{E}\left( {{{\left| {Y({\vartheta _{R, \hat{\delta}}})} \right|}^p}{{{\cal I}_{\{ {\vartheta _{R,\hat \delta }} \le T\} }}} } \right) = {R^p}\mathbb{P}({\vartheta _{R, \hat{\delta}}} \le T) ,
\end{split}
\end{equation*}
which implies 
\begin{equation*}
	\mathbb{P}({\vartheta _{R, \hat{\delta}}} \le T) \le \frac{{{C}}}{{{R^p}}}.
\end{equation*}
The proof is therefore complete.
\end{proof}

\subsection{$L^q$-strong convergence on finite time intervals}
In the following, our aim is to establish the strong convergence of ATS-TEM \eqref{scheme}-\eqref{eq-tem} in the $q$th moment sense with $ q \in (0, p)$, where $p$ is given in Assumption \ref{A2.1}. 

\begin{theorem} \label{strong-con2}
Let Assumption \ref{A2.1} hold. For any $q \in (0, p)$, 
\begin{equation*}
	\mathop {\lim }\limits_{{\check{\delta}} \to 0} \mathop {\sup }\limits_{0 \le t \le T} \mathbb{E}{\left| {X(t) -  Y(t)} \right|^q} = 0, \quad  
	\mathop {\lim }\limits_{{\check{\delta}} \to 0} \mathop {\sup }\limits_{0 \le t \le T} \mathbb{E}{\left| {X(t) - \bar Y(t)} \right|^q} = 0. 
\end{equation*}
\end{theorem}
\begin{proof}
We present the proof for $q \in [2, p)$. The case $q \in (0, 2)$ then follows from H\"older's inequality.
First, we define $\theta_{R, \hat{\delta}} = {\sigma _R} \wedge {\vartheta _{R, \hat{\delta}}}$, where $\sigma_R$ and $\vartheta_{R, \hat{\delta}}$ are defined in \eqref{sigmaR} and \eqref{varthetaR}, respectively. Recall Young's inequality, $ab \le \frac{\kappa }{r}{a^r} + \frac{1}{{s{\kappa ^{s/r}}}}{b^s}$ for any $1/r + 1/s = 1$ with $a, b, \kappa >0$ and $r, s > 1$. Then, for any $t \in [0, T]$,
\begin{equation}  \label{st1}
\begin{split}
	&\mathbb{E}{\left| {X(t) - {Y}(t)} \right|^q} \\
	=& \mathbb{E}\left( {{{\left| {X(t) - {Y}(t)} \right|}^q}{\mathcal{I}_{\{ \theta_{R, \hat{\delta}}  > T\} }}} \right) + \mathbb{E}\left( {{{\left| {X(t) - {Y}(t)} \right|}^q}{\mathcal{I}_{\{ {\theta_{R, \hat{\delta}}}  \le T\} }}} \right) \\
	\le& \mathbb{E}\left( {{{\left| {X(t \wedge {\theta_{R, \hat{\delta}}}) - Y(t \wedge {\theta_{R, \hat{\delta}}})} \right|}^q}{{\cal I}_{\{ {\theta_{R, \hat{\delta}}} > T\} }}} \right) + \frac{{q\kappa }}{p}\mathbb{E}{\left| {X(t) - {Y}(t)} \right|^p} + \frac{{p - q}}{{p{\kappa ^{q/(p - q)}}}}\mathbb{P}({\theta_{R, \hat{\delta}}} \le T) .
\end{split}
\end{equation}
According to Lemmas \ref{EXt}, \ref{lemma-stop} and Theorem \ref{th2.1}  one sees
\begin{equation*}
	\mathbb{E}{\left| {X(t) -  Y(t)} \right|^p} \le {2^{p - 1}}\mathop {\sup }\limits_{0 \le t \le T} \mathbb{E}{\left| {X(t)} \right|^p} + {2^{p - 1}}\mathop {\sup }\limits_{0 \le t \le T} \mathbb{E}{\left| { Y(t)} \right|^p} \le {C}
\end{equation*}
and
\begin{equation*} 
\begin{split}
	\mathbb{P}({\theta_{R, \hat{\delta}}} \le T) \le \mathbb{P}({\sigma _R} \le T) + \mathbb{P}({\vartheta _{R, \hat{\delta}} } \le T) \le \frac{{{C}}}{{{R^p}}}.
\end{split}
\end{equation*}
Substituting the above into \eqref{st1} yields
\begin{equation}  \label{st2}
\begin{split}
	\mathbb{E}{\left| {X(t) - {Y}(t)} \right|^q} \le& \mathbb{E}\left( {{{\left| {X(t \wedge {\theta_{R, \hat{\delta}}}) - Y(t \wedge {\theta_{R, \hat{\delta}}})} \right|}^q}{{\cal I}_{\{ {\theta_{R, \hat{\delta}}} > T\} }}} \right) + C\kappa + \frac{C}{{{\kappa ^{q/(p - q)}}{R^p}}} \\
	\le& \mathbb{E}\left( {{{\left| {X(t \wedge {\theta_{R, \hat{\delta}}}) -  Y(t \wedge {\theta_{R, \hat{\delta}}})} \right|}^q}} \right) + C\kappa + \frac{C}{{{\kappa ^{q/(p - q)}}{R^p}}}.
\end{split}
\end{equation}
Combining with \eqref{equation} and \eqref{continuous1}, applying the elementary inequality and \cite[Theorem 1.7.1]{mao2008stochastic} we have
\begin{equation}  \label{st4}
\begin{split}
	&\mathbb{E}{\left| {X(t \wedge {\theta_{R, \hat{\delta}}}) - Y(t \wedge {\theta_{R, \hat{\delta}}})} \right|^q} \\
	\le& {2^{q - 1}}\mathbb{E}{\left| {\int_0^{t \wedge {\theta_{R, \hat{\delta}}}} {\left( {f(X(s)) - F(\bar Y(s))} \right)ds} } \right|^q} + {2^{q - 1}}\mathbb{E}{\left| {\int_0^{t \wedge {\theta_{R, \hat{\delta}}}} {\left( {g(X(s)) - G(\bar Y(s))} \right)dW(s)} } \right|^q} \\
	\le& {2^{q - 1}}{T^{q - 1}}\mathbb{E}\left( {\int_0^{t \wedge {\theta_{R, \hat{\delta}}}} {{{\left| {f(X(s)) - F(\bar Y(s))} \right|}^q}ds} } \right)  \\
	&+ {2^{q - 1}}{\left( {\frac{{q(q - 1)}}{2}} \right)^{\frac{q}{2}}}{T^{\frac{{q - 2}}{2}}}\mathbb{E}\left( {\int_0^{t \wedge {\theta_{R, \hat{\delta}}}} {{{\left| {g(X(s)) - G(\bar Y(s))} \right|}^q}ds} } \right) \\
	\le& C\mathbb{E}\left( {\int_0^{t \wedge {\theta _{R,\hat \delta }}} {{{\left| {f(X(s)) - f(\bar Y(s))} \right|}^q}ds} } \right) + C\mathbb{E}\left( {\int_0^{t \wedge {\theta _{R,\hat \delta }}} {{{\left| {f(\bar Y(s)) - F(\bar Y(s))} \right|}^q}ds} } \right)  \\
	&+ C\mathbb{E}\left( {\int_0^{t \wedge {\theta _{R,\hat \delta }}} {{{\left| {g(X(s)) - g(\bar Y(s))} \right|}^q}ds} } \right) + C\mathbb{E}\left( {\int_0^{t \wedge {\theta _{R,\hat \delta }}} {{{\left| {g(\bar Y(s)) - G(\bar Y(s))} \right|}^q}ds} } \right) .
\end{split}
\end{equation}
It follows from \eqref{pi} and \eqref{FG} that
\begin{equation} \label{st8}
	\begin{split}
		&\left| {f(\bar Y(s)) - F(\bar Y(s))} \right| = \left| {f(\bar Y(s)) - {f_{\hat \delta }}(\bar Y(s))} \right|{{\cal I}_{\left\{ {{\delta _{{n_s}}} = \hat \delta } \right\}}} \\
		=& \left| {f(\bar Y(s)) - f\left( {{\pi _{\hat \delta }}(\bar Y(s))} \right)} \right|{{\cal I}_{\left\{ {\left| {\bar Y(s)} \right| > {\varphi ^{ - 1}}(K{{\hat \delta }^{ - \gamma }})} \right\}}}{{\cal I}_{\left\{ {{\delta _{{n_s}}} = \hat \delta } \right\}}} \\
		\le& \left| {f(\bar Y(s)) - f\left( {{\pi _{\hat \delta }}(\bar Y(s))} \right)} \right|{{\cal I}_{\left\{ {\left| {\bar Y(s)} \right| > {\varphi ^{ - 1}}(K{{\hat \delta }^{ - \gamma }})} \right\}}} .
	\end{split}
\end{equation}
Using Assumption \ref{llc} and H\"older's inequality we obtain
\begin{equation}   \label{st9}
	\begin{split}
			&\mathbb{E}\left( {\int_0^{t \wedge {\theta _{R,\hat \delta }}} {{{\left| {f(\bar Y(s)) - F(\bar Y(s))} \right|}^q}ds} } \right) \\
			\le& \mathbb{E}\left( {\int_0^T {\left| {f(\bar Y(s)) - f\left( {{\pi _{\hat \delta }}(\bar Y(s))} \right)} \right|^q{{\cal I}_{\left\{ {\left| {\bar Y(s)} \right| > {\varphi ^{ - 1}}(K{{\hat \delta }^{ - \gamma }})} \right\}}}ds} } \right) \\
			\le& {C_R} {\int_0^T {\mathbb{E}\left[ {{{\left| {\bar Y(s) - {\pi _{\hat \delta }}(\bar Y(s))} \right|}^q}{{\cal I}_{\left\{ {\left| {\bar Y(s)} \right| > {\varphi ^{ - 1}}(K{{\hat \delta }^{ - \gamma }})} \right\}}}} \right]ds} } \\
			\le& {C_R}\left( {\int_0^T {{{\left[ {\mathbb{E}{{\left| {\bar Y(s) - {\pi _{\hat \delta }}(\bar Y(s))} \right|}^p}} \right]}^{\frac{q}{p}}}{{\left[ {\mathbb{P}\left( {\left| {\bar Y(s)} \right| > {\varphi ^{ - 1}}(K{{\hat \delta }^{ - \gamma }})} \right)} \right]}^{\frac{{p - q}}{p}}}ds} } \right) .
	\end{split}
\end{equation}
From \eqref{pi} we see $\left| {{\pi _{\hat{\delta}} }(x)} \right| \le \left| x \right|$ for any $x \in \mathbb{R}^d$. Using Theorem \ref{th2.1}, one has
\begin{equation}  \label{st12}
	\begin{split}
		\mathbb{E}{\left| {\bar Y(s) - {\pi _{\hat \delta }}(\bar Y(s))} \right|^p} \le {2^{p - 1}}\left( {\mathbb{E}{{\left| {\bar Y(s)} \right|}^p} + \mathbb{E}{{\left| {{\pi _{\hat \delta }}(\bar Y(s))} \right|}^p}} \right) 
		\le {2^p}\mathop {\sup }\limits_{0 \le s \le T} \mathbb{E}{\left| {\bar Y(s)} \right|^p} \le C .
	\end{split}
\end{equation}
Besides, it follows from Chebyshev's inequality and Theorem \ref{th2.1} that
\begin{equation}  \label{st13}
	\begin{split}
		\mathbb{P}\left( {\left| {\bar Y(s)} \right| > {\varphi ^{ - 1}}(K{{\hat \delta }^{ - \gamma }})} \right) \le \frac{{\mathbb{E}{{\left| {\bar Y(s)} \right|}^p}}}{{{{\left( {{\varphi ^{ - 1}}(K{{\hat \delta }^{ - \gamma }})} \right)}^p}}} 
		\le C{\left( {{\varphi ^{ - 1}}(K{{\hat \delta }^{ - \gamma }})} \right)^{ - p}}.
	\end{split}
\end{equation}
Substituting \eqref{st12} and \eqref{st13} into \eqref{st9}, and noting that $0 < \hat{\delta} \le \check{\delta}$ as well as the monotonicity of ${\varphi ^{ - 1}}( \cdot )$, we deduce that 
\begin{equation}	\label{st10}
	\mathbb{E}\left( {\int_0^{t \wedge {\theta _{R,\hat \delta }}} {{{\left| {f(\bar Y(s)) - F(\bar Y(s))} \right|}^q}ds} } \right) \le {C_R}{\left( {{\varphi ^{ - 1}}(K{{\hat{\delta}} ^{ - \gamma }})} \right)^{ - (p - q)}} \le {C_R}{\left( {{\varphi ^{ - 1}}(K{{\check{\delta}} ^{ - \gamma }})} \right)^{ - (p - q)}}.
\end{equation}
Similar to \eqref{st8}-\eqref{st10}, one also sees
\begin{equation}  \label{st11}
\begin{split}
	\mathbb{E}\left( {\int_0^{t \wedge {\theta _{R,\hat \delta }}} {{{\left| {g(\bar Y(s)) - G(\bar Y(s))} \right|}^q}ds} } \right) \le {C_R}{\left( {{\varphi ^{ - 1}}(K{{\check{\delta}} ^{ - \gamma }})} \right)^{ - (p - q)}} .
\end{split}
\end{equation}
Hence, substituting \eqref{st10}, \eqref{st11} into \eqref{st4} yields
\begin{equation*}   
\begin{split}  
	&\mathbb{E}{\left| {X(t \wedge {\theta_{R, \hat{\delta}}}) - Y(t \wedge {\theta_{R, \hat{\delta}}})} \right|^q} \\
	\le& C\mathbb{E}\left( {\int_0^{t \wedge {\theta _{R,\hat \delta }}} {{{\left| {f(X(s )) - f(\bar Y(s))} \right|}^q}ds} } \right) + C\mathbb{E}\left( {\int_0^{t \wedge {\theta _{R,\hat \delta }}} {{{\left| {g(X(s)) - g(\bar Y(s))} \right|}^q}ds} } \right) \\
	&+ {C_R}{\left( {{\varphi ^{ - 1}}(K{\check{\delta} ^{ - \gamma}})} \right)^{ - (p - q)}}  .
\end{split}
\end{equation*}
Using Assumption \ref{A2.1} one has
\begin{equation}	\label{st15}
	\begin{split}
		&\mathbb{E}{\left| {X(t \wedge {\theta_{R, \hat{\delta}}}) - Y(t \wedge {\theta_{R, \hat{\delta}}})} \right|^q} \\
		\le& {C_R}\mathbb{E}\left( {\int_0^{t \wedge {\theta _{R,\hat \delta }}} {{{\left| {X(s ) - \bar Y(s)} \right|}^q}ds} } \right) + {C_R}{\left( {{\varphi ^{ - 1}}(K{\check{\delta} ^{ - \gamma}})} \right)^{ - (p - q)}} \\
		\le& {C_R}\mathbb{E}\left( {\int_0^t {{{\left| {X(s \wedge {\theta _{R,\hat \delta }}) - \bar Y(s \wedge {\theta _{R,\hat \delta }})} \right|}^q}ds} } \right) + {C_R}\int_0^T {\mathbb{E}{{\left| {Y(s) - \bar Y(s)} \right|}^q}ds}  \\
		&+ {C_R}{\left( {{\varphi ^{ - 1}}(K{\check{\delta} ^{ - \gamma}})} \right)^{ - (p - q)}} .
	\end{split}
\end{equation}
Similar to \eqref{kl1}, and using Theorem \ref{th2.1} yields
\begin{equation} \label{kl6}
\begin{split}
	\mathbb{E}{\left| {Y(t) - \bar Y(t)} \right|^q} \le  C{{\check{\delta} }^{\frac{q}{2}(1 - \gamma )}}\left( {1 + \mathbb{E}{{\left| {\bar Y(t)} \right|}^q}} \right) 
	\le C{{\check{\delta}}^{\frac{q}{2}(1 - \gamma )}} .
\end{split}
\end{equation}
Together with \eqref{st15} and \eqref{kl6} we obtain
\begin{equation*}
\begin{split}
	&\mathbb{E}{\left| {X(t \wedge {\theta_{R, \hat{\delta}}}) - Y(t \wedge {\theta_{R, \hat{\delta}}})} \right|^q} \\
	\le& {C_{R}}\int_0^t {\mathbb{E}{{\left| {X(s \wedge {\theta_{R, \hat{\delta}}}) - Y(s \wedge {\theta_{R, \hat{\delta}}})} \right|}^q}ds}  + {C_{R}}{{\check{\delta}} ^{{\frac{{q(1 - \gamma )}}{2}}}} + {C_R}{\left( {{\varphi ^{ - 1}}(K{\check{\delta} ^{ - \gamma}})} \right)^{ - (p - q)}} .
\end{split}
\end{equation*}
The Gronwall inequality gives that
\begin{equation*}
\begin{split}
	\mathbb{E}{\left| {X(t \wedge {\theta_{R, \hat{\delta}}}) -  Y(t \wedge {\theta_{R, \hat{\delta}}})} \right|^q} \le {C_R}\left[ {{{{\check{\delta}} ^{{\frac{{q(1 - \gamma )}}{2}}}}} + {\left( {{\varphi ^{ - 1}}(K{\check{\delta} ^{ - \gamma}})} \right)^{ - (p - q)}}} \right].
\end{split}
\end{equation*}
Substituting the above inequality into \eqref{st2} yields
\begin{equation*}
\begin{split}
	\mathbb{E}{\left| {X(t) - Y(t)} \right|^q} \le& {C_R}\left[ {{{\check{\delta}} ^{{\frac{{q(1 - \gamma )}}{2}}}} + {\left( {{\varphi ^{ - 1}}(K{\check{\delta} ^{ - \gamma}})} \right)^{ - (p - q)}}} \right] + C\kappa + \frac{C}{{{\kappa ^{q/(p - q)}}{R^p}}}.
\end{split}
\end{equation*}
For any given $\varepsilon > 0$, choose $\kappa(\varepsilon) > 0$ sufficiently small such that $C\kappa < \frac{\varepsilon }{3}$. Choose $R$ sufficiently large such that 
\begin{equation*}
	\frac{C}{{{\kappa ^{q/(p - q)}}{R^p}}} < \frac{\varepsilon }{3}.
\end{equation*}
Choose $\check{\delta} $ sufficiently small such that
\begin{equation*}  
	{C_R}\left[ {{{\check{\delta}} ^{{\frac{{q(1 - \gamma )}}{2}}}} + {\left( {{\varphi ^{ - 1}}(K{\check{\delta} ^{ - \gamma}})} \right)^{ - (p - q)}}} \right] < \frac{\varepsilon }{3}.
\end{equation*}
Therefore, 
\begin{equation*}
	\mathbb{E}{\left| {X(t) - Y(t)} \right|^q} < \varepsilon.
\end{equation*}
The first assertion follows. 
The second result follows from \eqref{kl6} that
\begin{equation*}
\begin{split}
	&\mathop {\lim }\limits_{{\check{\delta}} \to 0} \mathop {\sup }\limits_{0 \le t \le T} \mathbb{E}{\left| {X(t) -\bar Y(t)} \right|^q} \\
	\le& {2^{q - 1}}\mathop {\lim }\limits_{{\check{\delta}} \to 0} \mathop {\sup }\limits_{0 \le t \le T} \mathbb{E}{\left| {X(t) -  Y(t)} \right|^q} + {2^{q - 1}}\mathop {\lim }\limits_{{\check{\delta}} \to 0} \mathop {\sup }\limits_{0 \le t \le T} \mathbb{E}{\left| { Y(t) -\bar Y(t)} \right|^q} = 0 .
\end{split}
\end{equation*}
The proof is therefore complete.
\end{proof}

\section{Convergence rate in finite time interval}  \label{sec4}
In this section, we aim to investigate the convergence rate of the ATS-TEM \eqref{scheme}-\eqref{eq-tem} in a finite time interval $[0, T]$. We need somewhat stronger conditions than convergence alone, as stated below.

\begin{assumption} \label{A3.1}
There exist positive constants ${p^ * } >2$, $K_1, K_2$ and $l > 0$ such that for any $x, y \in \mathbb{R}^d$,
\begin{equation} \label{A4.1_1}
	\left\langle {x - y,f(x) - f(y)} \right\rangle  + \frac{{{p^*} - 1}}{2}{\left| {g(x) - g(y)} \right|^2} \le {K_1}{\left| {x - y} \right|^2},
\end{equation}
\begin{equation}   \label{f}
	\left| {f(x) - f(y)} \right| \le {K_2}\left( {1 + |x{|^l} + |y{|^l}} \right)\left| {x - y} \right|.
\end{equation}
\end{assumption}

\begin{remark}  \label{remark-2}
One notices that under Assumption \ref{A3.1}, for any $x, y \in \mathbb{R}^d$ we have
	\begin{equation}  \label{g}
		{\left| {g(x) - g(y)} \right|^2} \le K_3\left( {1 + |x{|^l} + |y{|^l}} \right){\left| {x - y} \right|^2},
	\end{equation}
	where $K_3 = (2(K_1 + K_2))/(p^*-1)$.
	Furthermore, by Young's inequality one sees
	\begin{equation}   \label{f-super}
		\left| {f(x)} \right| \le \left| {f(x) - f(0)} \right| + \left| {f(0)} \right| \le {K_2}\left( {1 + |x{|^l}} \right)\left| x \right| + \left| {f(0)} \right| \le C_f\left( {1 + |x{|^{l + 1}}} \right),
	\end{equation}
	\begin{equation}   \label{g-super}
		\left| {g(x)} \right| \le \left| {g(x) - g(0)} \right| + \left| {g(0)} \right| \le C{\left( {1 + |x{|^l}} \right)^{\frac{1}{2}}}\left| x \right| + \left| {g(0)} \right| \le C_g\left( {1 + |x{|^{\frac{l}{2} + 1}}} \right),
	\end{equation}
where $C_f = (2K_2) \vee |f(0)|$ and $C_g = (2K_3)^{1/2} \vee |g(0)|$.
\end{remark}

\begin{remark}  \label{remark3}
Under Assumption \ref{A3.1}, we may define $\varphi ( \cdot )$ in \eqref{phi} by 
$\varphi (r) = (1 \vee {C_f} \vee (2C_g^2))\left( {1 + {{\left| r \right|}^l}} \right)$ for any $r > 0$, then ${\varphi ^{ - 1}}(r) = \sqrt[l]{{r/(1 \vee {C_f} \vee (2C_g^2)) - 1}}$ for all $r \ge 1 \vee {C_f} \vee (2C_g^2)$. Then, the truncated mapping ${\pi _{\hat{\delta}} }( \cdot )$ can be chosen as
\begin{equation*}
	{\pi _{\hat{\delta}} }(x) = \left( {\left| x \right| \wedge \sqrt[l]{{\frac{{K{\hat{\delta}}^{ - \gamma}}}{{{C_f} \vee 2C_g^2}} - 1}}} \right)\frac{x}{{\left| x \right|}}, \quad \forall x \in \mathbb{R}^d,
\end{equation*}
where $\gamma \in (0, 1/3]$ will be specified in the proof of Theorem \ref{rate1}.
In addition, the adaptive timestep can be chosen as
\begin{equation*}
	{\delta _n} =    {\hat{\delta}}\vee\left(\frac{{{\check{\delta}}}}{{({C_f} \vee (2C_g^2))\left( {1 + {{\left| {{Y_{{t_n}}}} \right|}^l}} \right)}} \right), \quad n = 0,1,...
\end{equation*}
\end{remark}

\begin{lemma}  \label{baryy}
Let \eqref{A2.1-1} and  Assumption \ref{A3.1} hold with $p \ge l+2$. For any ${q_1}  \in [2, 2p/(l+2)]$, 
\begin{equation} \label{ij1}
	\mathop {\sup }\limits_{0 \le t \le T} \mathbb{E}{\left| {Y(t) - \bar Y(t)} \right|^{{q_1}}} \le {C}\check{\delta}^{\frac{{{q_1}}}{2}} .
\end{equation}

\end{lemma}
\begin{proof}  
From \eqref{continuous} and \eqref{continuous1} we see for any $t \in [0, T]$,
\begin{equation*}
	Y(t) - \bar Y(t) = Y(t) - {Y_{{t_{{n_t}}}}} = F({Y_{{t_{{n_t}}}}})(t - {t_{{n_t}}}) + G({Y_{{t_{{n_t}}}}})(W(t) - W({t_{{n_t}}})).
\end{equation*} 
Using \eqref{kl7} and Remark \ref{remark-1} one has
\begin{equation}  \label{ij}
\begin{split}
	&\mathbb{E}\left( {{{\left| {Y(t) - \bar Y(t)} \right|}^{{q_1}}}|{\mathcal{F}_{{t_{{n_t}}}}}} \right) \\
	\le& {2^{{q_1} - 1}}{\left| {F({Y_{{t_{{n_t}}}}})} \right|^{{q_1}}}{\left| {t - {t_{{n_t}}}} \right|^{{q_1}}} + {2^{{q_1} - 1}}{\left| {G({Y_{{t_{{n_t}}}}})} \right|^{{q_1}}}\mathbb{E}\left( {{{\left| {W(t) - W({t_{{n_t}}})} \right|}^{{q_1}}}|{{\cal F}_{{t_{{n_t}}}}}} \right) \\
	\le& {2^{{q_1} - 1}}{\left| {F({Y_{{t_{{n_t}}}}})} \right|^{{q_1}}}\check{\delta}^{{q_1}} + {2^{{q_1} - 1}}{\left| {G({Y_{{t_{{n_t}}}}})} \right|^{{q_1}}}\check{\delta}^{\frac{{{q_1}}}{2}} \quad a.s.
\end{split}
\end{equation}
From \eqref{pi} we see $\left| {{\pi _{\hat{\delta}} }(x)} \right| \le \left| x \right|$
for all $x \in \mathbb{R}^d$, using \eqref{FG} and \eqref{g-super} yields
\begin{equation*}
	\begin{split}
		{\left| {G({Y_{{t_{{n_t}}}}})} \right|^{{q_1}}} \le& {2^{{q_1} - 1}}{\left| {g({Y_{{t_{{n_t}}}}})} \right|^{{q_1}}} + {2^{{q_1} - 1}}{\left| {{g_{\hat{\delta}} }({Y_{{t_{{n_t}}}}})} \right|^{{q_1}}} \\
		=& {2^{{q_1} - 1}}{\left| {g({Y_{{t_{{n_t}}}}})} \right|^{{q_1}}} + {2^{{q_1} - 1}}{\left| {g({\pi _{\hat{\delta}} }({Y_{{t_{{n_t}}}}}))} \right|^{{q_1}}} \\
		\le& {2^{{q_1}}}{\left( {1 + {{\left| {{Y_{{t_{{n_t}}}}}} \right|}^{\frac{l}{2} + 1}}} \right)^{{q_1}}} .
	\end{split}
\end{equation*}
Substituting the above inequality into \eqref{ij}, using \eqref{Flinear} and the fact $\gamma \in (0, 1/3]$ one has
\begin{equation}  \label{ij2}
	\begin{split}
		&\mathbb{E}\left( {{{\left| {Y(t) - \bar Y(t)} \right|}^{{q_1}}}|{{\cal F}_{{t_{{n_t}}}}}} \right) \\
		\le& {2^{{q_1} - 1}}{(K + \rho )^{{q_1}}}{{\hat \delta }^{ - {q_1}\gamma }}{\left( {1 + \left| {{Y_{{t_{{n_t}}}}}} \right|} \right)^{{q_1}}}{{\check \delta }^{{q_1}}} + C{\left( {1 + {{\left| {{Y_{{t_{{n_t}}}}}} \right|}^{\frac{l}{2} + 1}}} \right)^{{q_1}}}{{\check \delta }^{\frac{{{q_1}}}{2}}} \\
		\le& C{{\check \delta }^{\frac{{2{q_1}}}{3}}}{\left( {1 + \left| {{Y_{{t_{{n_t}}}}}} \right|} \right)^{{q_1}}} + C{{\check \delta }^{\frac{{{q_1}}}{2}}}\left( {1 + {{\left| {{Y_{{t_{{n_t}}}}}} \right|}^{\frac{{{q_1}}}{2}(l + 2)}}} \right) \\
		\le& C{{\check \delta }^{\frac{{{q_1}}}{2}}}\left( {1 + {{\left| {{Y_{{t_{{n_t}}}}}} \right|}^{\frac{{{q_1}}}{2}(l + 2)}}} \right) , \quad a.s.
	\end{split}
\end{equation}
Taking expectations on both sides, using Lyapunov's inequality and Theorem \ref{th2.1} yields
\begin{equation*}
\begin{split}
	\mathbb{E}{\left| {Y(t) - \bar Y(t)} \right|^{{q_1}}} \le C{{\check \delta }^{\frac{{{q_1}}}{2}}}\left( {1 + \mathbb{E}{{\left| {{Y_{{t_{{n_t}}}}}} \right|}^{\frac{{{q_1}}}{2}(l + 2)}}} \right) \le C{{\check \delta }^{\frac{{{q_1}}}{2}}}\left[ {1 + {{\left( {\mathbb{E}{{\left| {{Y_{{t_{{n_t}}}}}} \right|}^p}} \right)}^{\frac{{{q_1}(l + 2)}}{{2p}}}}} \right] \le C{{\check \delta }^{\frac{{{q_1}}}{2}}} .
\end{split}
\end{equation*}
The proof is therefore complete.
\end{proof}

\begin{remark}
The requirements on the parameters $p$ and $q_1$ in Lemma \ref{baryy} are determined by the backstop numerical scheme. In the case where
	\begin{equation*}
		{\hat{\delta}} < {\delta _n} \le {\check{\delta}},\quad n = 0,1,..., \quad a.s.,
	\end{equation*} 
the backstop scheme in \eqref{scheme} is not activated. Consequently, by \eqref{FG}, $F({Y_{{t_{{n_t}}}}}) = f({Y_{{t_{{n_t}}}}})$, $G({Y_{{t_{{n_t}}}}}) = g({Y_{{t_{{n_t}}}}})$. Using \eqref{fg} and \eqref{eq2.14} we obtain
\begin{equation*}
\begin{split}
	&\mathbb{E}\left( {{{\left| {Y(t) - \bar Y(t)} \right|}^{{q_1}}}|{\mathcal{F}_{{t_{{n_t}}}}}} \right) \\
	\le& {2^{{q_1} - 1}}{\left| {f({Y_{{t_{{n_t}}}}})} \right|^{{q_1}}}\check{\delta}^{{q_1}} + {2^{{q_1} - 1}}{\left| {g({Y_{{t_{{n_t}}}}})} \right|^{{q_1}}}\check{\delta}^{\frac{{{q_1}}}{2}} \\
	\le& {2^{{q_1} - 1}}{\rho ^{{q_1}}}{\left( {1 + \left| {{Y_{{t_{{n_t}}}}}} \right|} \right)^{{q_1}}}\check{\delta}^{{q_1}} + {2^{{q_1} - 1}}{\rho ^{\frac{{{q_1}}}{2}}}{\left( {1 + \left| {{Y_{{t_{{n_t}}}}}} \right|} \right)^{{q_1}}}\check{\delta}^{\frac{{{q_1}}}{2}} \\
	\le& {2^{{q_1} - 1}}{\rho ^{{q_1}}}{\left( {1 + \left| {{Y_{{t_{{n_t}}}}}} \right|} \right)^{{q_1}}}\check{\delta}^{\frac{{{q_1}}}{2}} \quad a.s.
\end{split}
\end{equation*}
Taking expectations on both sides and applying Theorem \ref{th2.1} yield
\begin{equation*}
\begin{split}
	\mathbb{E}{\left| {Y(t) - \bar Y(t)} \right|^{{q_1}}} 
	\le C{{\check{\delta} }^{\frac{{{q_1}}}{2}}}\left( {1 + \mathbb{E}{{\left| {{Y_{{t_{{n_t}}}}}} \right|}^{{q_1}}}} \right) 
	\le  C\check{\delta}^{\frac{{{q_1}}}{2}} .
\end{split}
\end{equation*}
This argument only requires $p \ge 2$ and $q_1 \in [2, p]$, which is weaker than the assumptions in Lemma \ref{baryy}. Therefore, the conditions on $p$ and $q_1$ in Lemma \ref{baryy} are dictated by the choice of the backstop numerical scheme.
\end{remark}

\begin{theorem}   \label{rate1}
Let \eqref{A2.1-1} and Assumption \ref{A3.1} hold with $p \ge 2(l + 2) \vee (4l)$. Then for any $q \in [2,{p^*}) \cap [2,p/(l + 2)] \cap [2,p/(2l)]$, for the numerical solutions defined in \eqref{continuous} and \eqref{continuous1} with $\gamma = ql/(2(p-q))$,
\begin{equation*}
	\mathop {\sup }\limits_{0 \le t \le T} \mathbb{E}{\left| {X(t) - Y(t)} \right|^q} \le {C}\check{\delta}^{\frac{q}{2}},  \quad
	\mathop {\sup }\limits_{0 \le t \le T} \mathbb{E}{\left| {X(t) - \bar Y(t)} \right|^q} \le {C}\check{\delta}^{\frac{q}{2}}.
\end{equation*} 
\end{theorem}

\begin{proof}
For any $T > 0$, using Young's inequality to \eqref{continuous1}, we see that for any $t \in [0, T]$,
\begin{equation}  \label{jk}
	\begin{split}
		&\mathbb{E}{\left| {X(t) - {Y}(t)} \right|^q} \\
		=& \mathbb{E}\left( {{{\left| {X(t) - {Y}(t)} \right|}^q}{\mathcal{I}_{\{ \theta_{R, \hat{\delta}}  > T\} }}} \right) + \mathbb{E}\left( {{{\left| {X(t) - {Y}(t)} \right|}^q}{\mathcal{I}_{\{ {\theta_{R, \hat{\delta}}}  \le T\} }}} \right) \\
		\le& \mathbb{E}\left( {{{\left| {X(t \wedge {\theta_{R, \hat{\delta}}}) - Y(t \wedge {\theta_{R, \hat{\delta}}})} \right|}^q}{{\cal I}_{\{ {\theta_{R, \hat{\delta}}} > T\} }}} \right)  
		+\frac{q}{p}{\hat{\delta}}^{\frac{q}{2}}\mathbb{E}{\left| {X(t) - Y(t)} \right|^p} + \frac{{p - q}}{p}{\hat{\delta}}^{ - \frac{{{q^2}}}{{2(p - q)}}}\mathbb{P}({\theta_{R, \hat{\delta}}} \le T) .
	\end{split}
\end{equation}
According to Lemmas \ref{EXt}, \ref{lemma-stop} and Theorem \ref{th2.1}  we have
\begin{equation*}
	\mathbb{E}{\left| {X(t) -  Y(t)} \right|^p} \le {2^{p - 1}}\mathop {\sup }\limits_{0 \le t \le T} \mathbb{E}{\left| {X(t)} \right|^p} + {2^{p - 1}}\mathop {\sup }\limits_{0 \le t \le T} \mathbb{E}{\left| { Y(t)} \right|^p} \le {C}
\end{equation*}
and
\begin{equation*} 
	\begin{split}
		\mathbb{P}({\theta_{R, \hat{\delta}}} \le T) \le \mathbb{P}({\sigma _R} \le T) + \mathbb{P}({\vartheta _{R, \hat{\delta}} } \le T) \le \frac{{{C}}}{{{R^p}}}.
	\end{split}
\end{equation*}
Substituting the above into \eqref{jk} yields
\begin{equation}  \label{jk5}
\begin{split}
	\mathbb{E}{\left| {X(t) - Y(t)} \right|^q} \le \mathbb{E}{\left| {X(t \wedge {\theta_{R, \hat{\delta}}}) - Y(t \wedge {\theta_{R, \hat{\delta}}})} \right|^q} + {C}{\hat{\delta}}^{\frac{q}{2}} + \frac{C}{{{R^p}}}{{\hat \delta }^{ - \frac{{{q^2}}}{{2(p - q)}}}} .
\end{split}
\end{equation}
We may define $\varphi (r) = ({C_f} \vee (2C_g^2))\left( {1 + {{\left| r \right|}^l}} \right)$ for any $r > 0$, where $C_f, C_g, l$ are given in Remark \ref{remark-2}. Choosing $K = \varphi(1) = 2({C_f} \vee (2C_g^2))$,
	\begin{equation}  \label{jk6}
		R = {\varphi ^{ - 1}}(K{{\hat \delta }^{ - \gamma}}) = \sqrt[l]{{\frac{{K{{\hat \delta }^{ - \gamma }}}}{{{C_f} \vee (2C_g^2)}} - 1}}
	\end{equation}
	and $\gamma = ql/(2(p-q))$, these imply that $R \ge C{{\hat \delta }^{ - \frac{q}{{2(p - q)}}}}$. Hence, one sees ${{\hat \delta }^{ - \frac{{{q^2}}}{{2(p - q)}}}}/{R^p} \le C{{\hat \delta }^{\frac{q}{2}}}$.
Combining with \eqref{jk5} and the fact $0 < \hat{\delta} \le \check{\delta}$ one then has
\begin{equation}  \label{jk3}
\begin{split}
	\mathbb{E}{\left| {X(t) - Y(t)} \right|^q} \le& \mathbb{E}\left( {{{\left| {X(t \wedge {\theta_{R, \hat{\delta}}}) - Y(t \wedge {\theta_{R, \hat{\delta}}})} \right|}^q}} \right) + {C}{\hat{\delta}}^{\frac{q}{2}} \\
	\le& \mathbb{E}\left( {{{\left| {X(t \wedge {\theta_{R, \hat{\delta}}}) - Y(t \wedge {\theta_{R, \hat{\delta}}})} \right|}^q}} \right) + {C}\check{\delta}^{\frac{q}{2}} .
\end{split}
\end{equation}
On the other hand, it follows from \eqref{equation} and \eqref{continuous1} that
\begin{equation*}
\begin{split}
	X(t) - Y(t) = \int_0^t {\left[ {f(X(s)) - F(\bar Y(s))} \right]ds}  + \int_0^t {\left[ {g(X(s)) - G(\bar Y(s))} \right]dW(s)} .
\end{split}
\end{equation*}
The It\^o formula leads to
\begin{equation}   \label{jk1}
\begin{split}
	&\mathbb{E}{\left| {X(t \wedge {\theta_{R, \hat{\delta}}}) - Y(t \wedge {\theta_{R, \hat{\delta}}})} \right|^q} \\
	=& q\mathbb{E}\int_0^{t \wedge {\theta_{R, \hat{\delta}}}} {{{\left| {X(s) - Y(s)} \right|}^{q - 2}}\left[ {\left\langle {X(s) - Y(s),f(X(s)) - F(\bar Y(s))} \right\rangle  + \frac{{q - 1}}{2}{{\left| {g(X(s)) - G(\bar Y(s))} \right|}^2}} \right]ds} .
\end{split}
\end{equation}
Since $q \in [2, p^*)$, choose a small constant $\varsigma $ such that $(q - 1)(1 + \varsigma ) \le {p^*} - 1$. Then, using Assumption \ref{A3.1} one sees that for any $0 \le s \le t \wedge {\theta_{R, \hat{\delta}}}$,
\begin{equation*}  
	\begin{split} 
	&\left\langle {X(s) - Y(s),f(X(s)) - F(\bar Y(s))} \right\rangle  + \frac{{q - 1}}{2}{\left| {g(X(s)) - G(\bar Y(s))} \right|^2} \\
	=& \left\langle {X(s) - Y(s),f(X(s)) - f(Y(s))} \right\rangle  + \left\langle {X(s) - Y(s),f(Y(s)) - f(\bar Y(s))} \right\rangle \\
	&+ \left\langle {X(s) - Y(s),f(\bar Y(s)) - F(\bar Y(s))} \right\rangle  + \frac{{q - 1}}{2}(1 + \varsigma ){\left| {g(X(s)) - g(Y(s))} \right|^2} \\
	&+ (q - 1)\left( {1 + \frac{1}{\varsigma }} \right){\left| {g(Y(s)) - g(\bar Y(s))} \right|^2} + (q - 1)\left( {1 + \frac{1}{\varsigma }} \right){\left| {g(\bar Y(s)) - G(\bar Y(s))} \right|^2} \\
	\le& {K_1}{\left| {X(s) - Y(s)} \right|^2} + \left| {X(s) - Y(s)} \right|\left| {f(Y(s)) - f(\bar Y(s))} \right| + \left| {X(s) - Y(s)} \right|\left| {f(\bar Y(s)) - F(\bar Y(s))} \right| \\
	&+ (q - 1)\left( {1 + \frac{1}{\varsigma }} \right){\left| {g(Y(s)) - g(\bar Y(s))} \right|^2} + (q - 1)\left( {1 + \frac{1}{\varsigma }} \right){\left| {g(\bar Y(s)) - G(\bar Y(s))} \right|^2} .
\end{split}
\end{equation*}
Inserting the above inequality into \eqref{jk1}, and applying Young's inequality one gets
\begin{equation}   \label{st6}
\begin{split}
	&\mathbb{E}{\left| {X(t \wedge {\theta_{R, \hat{\delta}}}) - Y(t \wedge {\theta_{R, \hat{\delta}}})} \right|^q} \\
	\le& C\mathbb{E}\left( {\int_0^{t \wedge {\theta_{R, \hat{\delta}}}} {{{\left| {X(s) - Y(s)} \right|}^q}ds} } \right) + C\mathbb{E}\left( {\int_0^T {{{\left| {f(Y(s)) - f(\bar Y(s))} \right|}^q}ds} } \right) \\
	& + C\mathbb{E}\left( {\int_0^{t \wedge {\theta_{R, \hat{\delta}}}} {{{\left| {f(\bar Y(s)) - F(\bar Y(s))} \right|}^q}ds} } \right) + C\mathbb{E}\left( {\int_0^T {{{\left| {g(Y(s)) - g(\bar Y(s))} \right|}^q}ds} } \right) \\
	&+ C\mathbb{E}\left( {\int_0^{t \wedge {\theta_{R, \hat{\delta}}}} {{{\left| {g(\bar Y(s)) - G(\bar Y(s))} \right|}^q}ds} } \right).
\end{split}
\end{equation}
It follows from \eqref{jk6} that $R = {\varphi ^{ - 1}}(K{\hat{\delta}}^{ - \gamma})$. Then, for any $s \in [0,t \wedge {\vartheta _{R, \hat{\delta}}}]$, using \eqref{pi} and \eqref{varthetaR} we have $f(\bar Y(s)) = {f_{\hat{\delta}} }(\bar Y(s))$ and $g(\bar Y(s))  = {g_{\hat{\delta}} }(\bar Y(s))$.
Applying \eqref{FG} one gets
\begin{equation} \label{st5}
\begin{split}
	\mathbb{E}\left( {\int_0^{t \wedge {\theta _{R,\hat \delta }}} {{{\left| {f(\bar Y(s)) - F(\bar Y(s))} \right|}^q}ds} } \right) \le& \mathbb{E}\left( {\int_0^{t \wedge {\vartheta _{R,\hat \delta }}} {{{\left| {f(\bar Y(s)) - {f_{\hat \delta }}(\bar Y(s))} \right|}^q}{{\cal I}_{\{ {\delta _{{n_s}}} = \hat \delta \} }}ds} } \right) = 0, \\
	\mathbb{E}\left( {\int_0^{t \wedge {\theta _{R,\hat \delta }}} {{{\left| {g(\bar Y(s)) - G(\bar Y(s))} \right|}^q}ds} } \right) \le& \mathbb{E}\left( {\int_0^{t \wedge {\vartheta _{R,\hat \delta }}} {{{\left| {g(\bar Y(s)) - {g_{\hat \delta }}(\bar Y(s))} \right|}^q}{{\cal I}_{\{ {\delta _{{n_s}}} = \hat \delta \} }}ds} } \right) = 0 .
\end{split}
\end{equation}
Inserting \eqref{st5} into \eqref{st6}, using \eqref{f} and \eqref{g}  yields
\begin{equation}  \label{jk2}
\begin{split}
	&\mathbb{E}{\left| {X(t \wedge {\theta_{R, \hat{\delta}}}) - Y(t \wedge {\theta_{R, \hat{\delta}}})} \right|^q} \\
	\le& C\mathbb{E}\left[ {\int_0^t {{{\left| {X(s \wedge {\theta_{R, \hat{\delta}}}) - Y(s \wedge {\theta_{R, \hat{\delta}}})} \right|}^q}ds} } \right] + C\mathbb{E}\left[ {\int_0^T {{{\left( {1 + {{\left| {Y(s)} \right|}^l} + {{\left| {\bar Y(s)} \right|}^l}} \right)}^q}{{\left| {Y(s) - \bar Y(s)} \right|}^q}ds} } \right] \\
	&+C\mathbb{E}\left[ {\int_0^T {{{\left( {1 + {{\left| {Y(s)} \right|}^l} + {{\left| {\bar Y(s)} \right|}^l}} \right)}^{\frac{q}{2}}}{{\left| {Y(s) - \bar Y(s)} \right|}^q}ds} } \right] \\
	\le& C\int_0^t {\mathbb{E}{{\left| {X(s \wedge {\theta_{R, \hat{\delta}}}) - Y(s \wedge {\theta_{R, \hat{\delta}}})} \right|}^q}ds}  + C\int_0^T {\mathbb{E}\left[ {{{\left( {1 + {{\left| {Y(s)} \right|}^l} + {{\left| {\bar Y(s)} \right|}^l}} \right)}^q}{{\left| {Y(s) - \bar Y(s)} \right|}^q}} \right]ds} .
\end{split}
\end{equation}
Using H\"older's inequality and Jensen's inequality, combining with Theorem \ref{th2.1} and Lemma \ref{baryy} and using the fact $q \in [2,p/(l + 2)] \cap [2,p/(2l)]$  yields 
\begin{equation}  \label{lem-con-2}
	\begin{split}
		&\mathbb{E}\left[ {{{\left( {1 + {{\left| {\bar Y(s)} \right|}^l} + {{\left| {Y(s)} \right|}^l}} \right)}^q}{{\left| {\bar Y(s) - Y(s)} \right|}^q}} \right] \\
		\le& C\mathbb{E}\left[ {\left( {1 + {{\left| {\bar Y(s)} \right|}^{lq}} + {{\left| {Y(s)} \right|}^{lq}}} \right){{\left| {\bar Y(s) - Y(s)} \right|}^q}} \right] \\
		\le& C{\left[ {\mathbb{E}{{\left( {1 + {{\left| {\bar Y(s)} \right|}^{lq}} + {{\left| {Y(s)} \right|}^{lq}}} \right)}^2}} \right]^{\frac{1}{2}}}{\left[ {\mathbb{E}{{\left| {\bar Y(s) - Y(s)} \right|}^{2q}}} \right]^{\frac{1}{2}}} \\
		\le& C{\left[ {1 + \mathbb{E}{{\left| {\bar Y(s)} \right|}^{2lq}} + \mathbb{E}{{\left| {Y(s)} \right|}^{2lq}}} \right]^{\frac{1}{2}}}{\left[ {\mathbb{E}{{\left| {\bar Y(s) - Y(s)} \right|}^{2q}}} \right]^{\frac{1}{2}}} \\
		\le& C\left[ {1 + {{\left( {\mathbb{E}{{\left| {\bar Y(s)} \right|}^p}} \right)}^{\frac{{2lq}}{p}}} + {{\left( {\mathbb{E}{{\left| {Y(s)} \right|}^p}} \right)}^{\frac{{2lq}}{p}}}} \right]^{\frac{1}{2}}{\left[ {\mathbb{E}{{\left| {\bar Y(s) - Y(s)} \right|}^{\frac{2p}{{l + 2}}}}} \right]^{\frac{{(l + 2)q}}{p}}} \\
		\le& {C}\check{\delta}^{\frac{q}{2}} .
	\end{split}
\end{equation}
Combining this with \eqref{jk2} and applying Gronwall's inequality one has
\begin{equation*}
	\mathbb{E}{\left| {X(t \wedge {\theta_{R, \hat{\delta}}}) - Y(t \wedge {\theta_{R, \hat{\delta}}})} \right|^q} \le {C}\check{\delta}^{\frac{q}{2}} .
\end{equation*}
Substituting the above inequality into \eqref{jk3} we obtain 
\begin{equation*}
\begin{split}
	\mathbb{E}{\left| {X(t) - Y(t)} \right|^q} \le {C}\check{\delta}^{\frac{q}{2}}. 
\end{split}
\end{equation*}
The first assertion follows. The second result follows from Lemma \ref{baryy},
\begin{equation*}
	\mathop {\sup }\limits_{0 \le t \le T} \mathbb{E}{\left| {X(t) - \bar Y(t)} \right|^q} \le {2^{q - 1}}\mathop {\sup }\limits_{0 \le t \le T}\mathbb{E}{\left| {X(t) -  Y(t)} \right|^q} + {2^{q - 1}}\mathop {\sup }\limits_{0 \le t \le T}\mathbb{E}{\left| { Y(t) - \bar Y(t)} \right|^q} \le {C}\check{\delta}^{\frac{q}{2}}.
\end{equation*}
The proof is therefore complete.
\end{proof}

\section{Other choices of the backstop scheme}  \label{sec5}
The TEM scheme is not the unique choice of the backstop scheme. In fact, 
we may choose the TaEM scheme introduced in \cite{sabanis2016euler} as our backstop scheme. We prove that TaEM has several nice properties and the  ATS-TaEM scheme achieves the strong convergence rate with order $1/2$.
We choose TaEM as our backstop scheme (Model 2, see \cite{sabanis2016euler}), that is, 
\begin{equation*}
	U({Y_{{t_n}}},{\delta _n}, \Delta {W_n}) = \tilde f({Y_{{t_n}}}){\delta _n} + \tilde g({Y_{{t_n}}})\Delta {W_n}, \quad n = 0,1,...,
\end{equation*}
where $\delta_n$ is given in \eqref{timestep}, and 
\begin{equation}  \label{tildefg}
	\tilde f(x) = \frac{{f(x)}}{{1 + {{\hat \delta }^{1/2}}{{\left| x \right|}^l}}}, \quad \tilde g(x) = \frac{{g(x)}}{{1 + {{\hat \delta }^{1/2}}{{\left| x \right|}^l}}}, \quad \forall x \in \mathbb{R}^d.
\end{equation}
and $l \ge 0$ is given in Assumption \ref{A3.1}. One has $|\tilde f(x)| \le |f(x)| $ and $|\tilde g(x)| \le |g(x)| $ for any $x \in \mathbb{R}^d$. The ATS-TaEM is defined by
\begin{equation}  \label{ATS-TaEM}
	\left\{  
	\begin{array}{l}
		Y_0 = x_0, \quad t_0 = 0, \\
		\delta_n := \delta(Y_{t_n}), \quad t_{n+1} = t_n + \delta_n, \quad n = 0,1,...,  \\
		{Y_{{t_{n + 1}}}} = {Y_{{t_n}}} + \left\{ {f({Y_{{t_n}}}){\delta _n} + g({Y_{{t_n}}})\Delta {W_n}} \right\}{\mathcal{I}_{\{ \hat \delta  < {\delta _n} \le \check{\delta} \} }}   \\
		\quad + \left\{ {\tilde f({Y_{{t_n}}}){\delta _n} + \tilde g({Y_{{t_n}}})\Delta {W_n}} \right\}{{\cal I}_{\{ {\delta _n} = \hat \delta \} }}, \quad n = 0,1,...,
	\end{array}
	\right.
\end{equation}
Define two versions of the continuous-time numerical solutions by
\begin{equation} \label{continuous-3}
	\bar Y(t) = Y_{t_{n_t}}, \quad  \forall t \ge 0,
\end{equation}
and 
\begin{equation} \label{continuous-2}
	{Y}(t)= {x_0} + \int_0^t {\tilde F(\bar Y(s))ds}  + \int_0^t {\tilde G(\bar Y(s))dW(s)} , \quad \forall t \ge 0,
\end{equation}
where 
\begin{equation}  \label{FG1}
	\begin{split}
		\tilde F(x) =& f(x)\mathcal{I}_{\{ \hat{\delta}  < \delta (x) \le \check{\delta} \}}  + \tilde f(x){{\cal I}_{\{ \delta (x) = \hat \delta \} }} ,\\
		\tilde G(x) =& g(x)\mathcal{I}_{\{ \hat{\delta}  < \delta (x) \le \check{\delta} \}}  + \tilde g(x){{\cal I}_{\{ \delta (x) = \hat \delta \} }}. 
	\end{split}
\end{equation}

In the following, we see that $\tilde F(\cdot)$ and $\tilde G(\cdot)$ have some nice properties.
Suppose that Assumption \ref{A3.1} holds. Using \eqref{tildefg} and Remark \ref{remark-2} one computes
\begin{equation}  \label{flinear}
	\begin{split}
		\left| {\tilde f(x)} \right| = \frac{{\left| {f(x)} \right|}}{{1 + {\hat{\delta} ^{1/2}}{{\left| x \right|}^l}}} \le \frac{{{C_f}\left( {1 + {{\left| x \right|}^{l + 1}}} \right)}}{{1 + {\hat{\delta} ^{1/2}}{{\left| x \right|}^l}}} \le {C_f}{\hat{\delta} ^{ - \frac{1}{2}}}\left( {1 + \left| x \right|} \right).
	\end{split}
\end{equation}
and 
\begin{equation}   \label{glinear}
	\begin{split}
		{\left| {\tilde g(x)} \right|^2} \le \frac{{C_g^2{{\left( {1 + {{\left| x \right|}^{\frac{l}{2} + 1}}} \right)}^2}}}{{{{\left( {1 + {\hat{\delta} ^{1/2}}{{\left| x \right|}^l}} \right)}^2}}} \le \frac{{C_g^2\left( {1 + {{\left| x \right|}^2}} \right)\left( {1 + {{\left| x \right|}^l}} \right)}}{{{{\left( {1 + {\hat{\delta} ^{1/2}}{{\left| x \right|}^l}} \right)}^2}}} \le C_g^2{\hat{\delta} ^{ - \frac{1}{2}}}\left( {1 + {{\left| x \right|}^2}} \right) .
	\end{split}
\end{equation}
Hence, using these together with \eqref{FG1}, \eqref{fg} and \eqref{eq2.14} we obtain
\begin{equation}  \label{FGlinear}
	\left| {\tilde F(x)} \right| \le \left( {{C_f} + \rho } \right){{\hat \delta }^{ - \frac{1}{2}}}\left( {1 + \left| x \right|} \right),\quad
	\left| {\tilde G(x)} \right| \le \left( {{C_g^2} + {\rho }^{\frac{1}{2}}} \right){{\hat \delta }^{ - \frac{1}{4}}}\left( {1 + \left| x \right|} \right). 
\end{equation}
Furthermore, suppose that \eqref{A2.1-1} holds, then we see
\begin{equation}   \label{fg-Khas}
	\begin{split}
		&\left\langle {x,{{\tilde f} }(x)} \right\rangle  + \frac{{p - 1}}{2}{\left| {{{\tilde g} }(x)} \right|^2} 
		\le \frac{1}{{1 + {\hat{\delta} ^{1/2}}{{\left| x \right|}^l}}}\left[ {\left\langle {x,f(x)} \right\rangle  + \frac{{p - 1}}{2}{{\left| {g(x)} \right|}^2}} \right] \\
		\le& \left\langle {x,f(x)} \right\rangle  + \frac{{p - 1}}{2}{\left| {g(x)} \right|^2}
		\le \alpha \left( {1 + {{\left| x \right|}^2}} \right) .
	\end{split}
\end{equation}
Hence, similar to Lemma \ref{fgdelta} one has
\begin{equation}  \label{FG-Khas}
	\left\langle {x,\tilde F(x)} \right\rangle  + \frac{{p - 1}}{2}{\left| {\tilde G(x)} \right|^2} \le \alpha\left( {1 + {{\left| x \right|}^2}} \right) .
\end{equation}
That is, $\tilde{F}(\cdot)$ and $\tilde G(\cdot)$ preserve the Khasminskii-type condition \eqref{A2.1-1}.

With \eqref{FGlinear} and \eqref{FG-Khas} in hand, similar to the proof of Theorem \ref{th2.1} and Lemma \ref{baryy}, we obtain the following two results. 

\begin{lemma}	\label{EZt}
Let Assumptions \ref{A2.1} and \ref{A3.1} hold. For any $ T > 0$, there exists a positive constant $C$ such that
\begin{equation*} 
	\mathop {\sup }\limits_{0 \le t \le T} \mathbb{E}{\left| {\bar Y(t)} \right|^p} \le C, \quad	\mathop {\sup }\limits_{0 \le t \le T} \mathbb{E}{\left| {Y(t)} \right|^p} \le C,
\end{equation*} 
where $p$ is given in Assumption \ref{A2.1}.
\end{lemma}

\begin{lemma}	\label{lm5}
Let \eqref{A2.1-1} and Assumptios \ref{A3.1} hold with $p \ge l+2$. For any $q_2 \in [2, 2p/(l+2)]$,
\begin{equation*}  
	\mathop {\sup }\limits_{0 \le t \le T} \mathbb{E}{\left| {Y(t) - \bar Y(t)} \right|^{{q_2}}} \le C\check{\delta}^{\frac{{{q_2}}}{2}}  \quad \forall T > 0.
\end{equation*}
\end{lemma}

In the following, we consider the convergence rate of the ATS-TaEM \eqref{ATS-TaEM}.

\begin{theorem}  \label{th-5.3}
Let \eqref{A2.1-1} and Assumption \ref{A3.1} hold with $p \ge 2(l + 2) \vee (4l+2) $, for any $q \in [2,{p^*}) \cap [2,p/(l + 2)] \cap [2,p/(2l+1)] $, the numerical solutions defined by \eqref{continuous-2} and \eqref{continuous-3} satisfy
\begin{equation*}
	\mathop {\sup }\limits_{0 \le t \le T} \mathbb{E}{\left| {X(t) - Y(t)} \right|^q} \le {C}\check{\delta}^{\frac{q}{2}},  \quad
	\mathop {\sup }\limits_{0 \le t \le T} \mathbb{E}{\left| {X(t) - \bar Y(t)} \right|^q} \le {C}\check{\delta}^{\frac{q}{2}}.
\end{equation*}
\end{theorem}

\begin{proof}
Fix $T > 0$. Similar to \eqref{jk1}-\eqref{st6}, applying It\^o's formula to \eqref{continuous-2}, Assumption \ref{A3.1} and Young's inequality, for any $t \in [0, T]$, one gets
\begin{equation}  \label{lm10}
	\begin{split}
		&\mathbb{E}{\left| {X(t) - Y(t)} \right|^q} \\
		\le& C\mathbb{E}\left( {\int_0^{t} {{{\left| {X(s) - Y(s)} \right|}^q}ds} } \right) + C\mathbb{E}\left( {\int_0^t {{{\left| {f(Y(s)) - f(\bar Y(s))} \right|}^q}ds} } \right) \\
		&+ C\mathbb{E}\left( {\int_0^{t} {{{\left| {f(\bar Y(s)) - \tilde F(\bar Y(s))} \right|}^q}ds} } \right) + C\mathbb{E}\left( {\int_0^t {{{\left| {g(Y(s)) - g(\bar Y(s))} \right|}^q}ds} } \right) \\
		& + C\mathbb{E}\left( {\int_0^{t} {{{\left| {g(\bar Y(s)) - \tilde G(\bar Y(s))} \right|}^q}ds} } \right) .
	\end{split}
\end{equation}
Using \eqref{FG1}, \eqref{tildefg} and Assumption \ref{A3.1} we compute
\begin{equation*}
\begin{split}
	&\left| {f(\bar Y(s)) - \tilde F(\bar Y(s))} \right| = \left| {f(\bar Y(s)) - \tilde f(\bar Y(s))} \right|{\mathcal{I}_{\{ {\delta _{{n_s}}} = \hat \delta \} }} \\
	\le& \left| {f(\bar Y(s))} \right|\frac{{{{\hat \delta }^{1/2}}{{\left| {\bar Y(s)} \right|}^l}}}{{1 + {{\hat \delta }^{1/2}}{{\left| {\bar Y(s)} \right|}^l}}} 
	\le {C_f}{{\hat \delta }^{\frac{1}{2}}}{\left| {\bar Y(s)} \right|^l}\left( {1 + {{\left| {\bar Y(s)} \right|}^{l + 1}}} \right) \\
	\le& C{{\hat \delta }^{\frac{1}{2}}}\left( {1 + {{\left| {\bar Y(s)} \right|}^{2l + 1}}} \right) 
\end{split}
\end{equation*}
Then, using the facts $0< \hat{\delta} \le \check{\delta}$ and $q \in [2, p/(2l+1)]$ yields
\begin{equation*}  
\begin{split}
	&\mathbb{E}\left( {\int_0^t {{{\left| {f(\bar Y(s)) - \tilde F(\bar Y(s))} \right|}^q}ds} } \right) \le C{{\hat \delta }^{\frac{q}{2}}}\int_0^T {\mathbb{E}{{\left( {1 + {{\left| {\bar Y(t)} \right|}^{2l + 1}}} \right)}^q}dt}  \\
	\le& C{{\check \delta }^{\frac{q}{2}}}\int_0^T {\left( {1 + \mathbb{E}{{\left| {\bar Y(t)} \right|}^{q(2l + 1)}}} \right)dt}  \\
	\le& C{{\check \delta }^{\frac{q}{2}}}\int_0^T {\left( {1 + {{\left( {\mathbb{E}{{\left| {\bar Y(t)} \right|}^p}} \right)}^{\frac{{q(2l + 1)}}{p}}}} \right)dt} 	\le  C{{\check \delta }^{\frac{q}{2}}} .
\end{split}
\end{equation*}
Similarly, one has
\begin{equation*}  
	\mathbb{E}\left( {\int_0^t {{{\left| {g(\bar Y(s)) - \tilde G(\bar Y(s))} \right|}^q}ds} } \right) \le C{{\check \delta }^{\frac{q}{2}}}.
\end{equation*}
Inserting the above into \eqref{lm10}, using Assumption \ref{A3.1} yields
\begin{equation*}
	\begin{split}
		&\mathbb{E}{\left| {X(t) - Y(t)} \right|^q} \\
		\le& C\mathbb{E}\left( {\int_0^t {{{\left| {X(s) - Y(s)} \right|}^q}ds} } \right) + C\mathbb{E}\left[ {\int_0^t {{{\left( {1 + {{\left| {Y(s)} \right|}^l} + {{\left| {\bar Y(s)} \right|}^l}} \right)}^q}{{\left| {Y(s) - \bar Y(s)} \right|}^q}ds} } \right] \\
		&+ C\mathbb{E}\left( {\int_0^t {{{\left( {1 + {{\left| {Y(s)} \right|}^l} + {{\left| {\bar Y(s)} \right|}^l}} \right)}^{\frac{q}{2}}}{{\left| {Y(s) - \bar Y(s)} \right|}^q}ds} } \right)  + C\check{\delta}^{\frac{q}{2}} \\
		\le& C\mathbb{E}\left( {\int_0^t {{{\left| {X(s) - Y(s)} \right|}^q}ds} } \right)  + C\int_0^t {\mathbb{E}\left[ {{{\left( {1 + {{\left| {Y(s)} \right|}^l} + {{\left| {\bar Y(s)} \right|}^l}} \right)}^q}{{\left| {Y(s) - \bar Y(s)} \right|}^q}} \right]ds} 
		+ C\check{\delta}^{\frac{q}{2}} .
	\end{split}
\end{equation*}
Similar to \eqref{lem-con-2}, since $q \in [2,p/(l + 2)] \cap [2,p/(2l+1)]$ we have
\begin{equation*}
\begin{split}
	{\mathbb{E}\left[ {{{\left( {1 + {{\left| {Y(s)} \right|}^l} + {{\left| {\bar Y(s)} \right|}^l}} \right)}^q}{{\left| {Y(s) - \bar Y(s)} \right|}^q}} \right]} \le C\check{\delta}^{\frac{q}{2}} .
\end{split}
\end{equation*}
Hence, one gets
\begin{equation*}
	\begin{split}
		\mathbb{E}{\left| {X(t) - Y(t)} \right|^q} \le C\int_0^t {\mathbb{E}{{\left| {X(s) - Y(s )} \right|}^q}ds}   + C\check{\delta}^{\frac{q}{2}}.
	\end{split}
\end{equation*}
The Gronwall inequality gives
\begin{equation*}
	\mathbb{E}{\left| {X(t) - Y(t)} \right|^q} \le {C}\check{\delta}^{\frac{q}{2}}, \quad \forall t \in [0, T].
\end{equation*}
Combining with Lemma \ref{lm5} we then have
\begin{equation*}
	\mathop {\sup }\limits_{0 \le t \le T} \mathbb{E}{\left| {X(t) - \bar Y(t)} \right|^q} \le {2^{q - 1}}\mathop {\sup }\limits_{0 \le t \le T} \mathbb{E}{\left| {X(t) - Y(t)} \right|^q} + {2^{q - 1}}\mathop {\sup }\limits_{0 \le t \le T} \mathbb{E}{\left| {Y(t) - \bar Y(t)} \right|^q} \le C{\check{\delta} ^{\frac{q}{2}}} .
\end{equation*}
The proof is therefore complete.
\end{proof}

Consequently, the TaEM proposed in \cite[p.2087, Model 2]{sabanis2016euler} can be employed as the backstop scheme, in which case the corresponding ATS-TaEM attains the optimal  $1/2$-order strong convergence rate. Moreover, Model 1 in \cite{sabanis2016euler} can also be adopted as a backstop scheme. In this case, the modified drift and diffusion coefficients are given by
\begin{equation} \label{model1-fg}
	\tilde f(x) = \frac{{f(x)}}{{1 + {{\hat \delta }^{1/2}}\left| {f(x)} \right| + {{\hat \delta }^{1/2}}{{\left| {g(x)} \right|}^2}}}, \quad \tilde g(x) = \frac{{g(x)}}{{1 + {{\hat \delta }^{1/2}}\left| {f(x)} \right| + {{\hat \delta }^{1/2}}{{\left| {g(x)} \right|}^2}}}, \quad \forall x \in \mathbb{R}^d.
\end{equation}
It is straightforward to verify that the conditions \eqref{flinear}, \eqref{glinear}, and \eqref{fg-Khas} remain valid for the coefficients in \eqref{model1-fg}. The $1/2$-order strong convergence rate can be obtained using similar techniques to those in Theorem \ref{th-5.3} and is therefore omitted to avoid duplication.

\section{Numerical experiments}  \label{sec6}
In the numerical experiments below, we compare the ATS-TEM scheme \eqref{scheme}-\eqref{eq-tem} to a number of different fixed-step numerical schemes. For the latter, we take as the uniform stepsize $\delta_{mean}$ the average of all time steps $\delta_n^{(m)}$ over each path and each realization $m = 1,..., M$ so that
\begin{equation}
	{\delta_{mean}} = \frac{1}{M}\sum\limits_{m = 1}^M {\frac{1}{{{N^{(m)}}}}\sum\limits_{n = 1}^{{N^{(m)}}} {\delta _n^{(m)}} }. 
\end{equation}
This implies that the ATS-TEM and the fixed-step numerical schemes require the same number of timesteps to reach the terminal time $T$. Consequently, their computational costs are comparable, which allows for a fair comparison of their numerical performance \cite{kelly2018adaptive}.
In the following, errors are measured in terms of the root mean square error (RMSE) at the terminal time $T = 1$. The expectations are approximated by computing averages over $M = 500$ sample paths. The numerical experiments are tested by Matlab R2023a on DESKTOP-QRSSSIC (11VECTO1WW).

In the following, we first introduce three widely used fixed-step numerical schemes with uniform stepsize $\delta_{mean}$.

\begin{itemize}
\item[(1)] \textbf{Truncated EM scheme} (see \cite{li2019explicit})
    The fixed step truncated EM scheme (\textbf{FS-TEM}) is described by 
\begin{equation}   \label{TEMa}
	\begin{split}
		\left\{\begin{array}{l}
			{{Y_0} = {x_0},} \\
			{{{\tilde Y}_{{t_{n + 1}}}} = {Y_{{t_n}}} + f({Y_{{t_n}}}){\delta _{mean}} + g({Y_{{t_n}}})(W((n + 1){\delta _{mean}}) - W(n{\delta _{mean}}))}, \\
			{{Y_{{t_{ n+1}}}} = \left( {\left| {{{\tilde Y}_{{t_{n + 1}}}}} \right| \wedge {\varphi ^{ - 1}}(K\delta _{mean}^{ - \gamma })} \right)\frac{{{{\tilde Y}_{{t_{n + 1}}}}}}{{\left| {{{\tilde Y}_{{t_{n + 1}}}}} \right|}}} 
		\end{array}\right.
	\end{split}
\end{equation}
with $n = 0,1,...$, where $\varphi :{\mathbb{R}_ + } \to {\mathbb{R}_ + }$ is a strictly increasing continuous function defined in \eqref{phi}, ${\varphi ^{ - 1}}$ represents the inverse function of $\varphi$, which is ${\varphi ^{ - 1}}:[\varphi (0), + \infty ) \to {\mathbb{R}_ + }$ is a strictly increasing continuous function, $K \ge \left| {f(0)} \right| \vee {\left| {g(0)} \right|^2} \vee \varphi (1)$ is a positive constant independent of $\delta_{mean}$, and $\gamma \in (0, 1/2]$.

\item[(2)] \textbf{Tamed EM scheme} (see \cite{sabanis2016euler}) The fixed-step tamed EM scheme (\textbf{FS-TaEM}) is described by 
\begin{equation}   \label{TaEM}
	\begin{split}
		\left\{\begin{array}{l}
			{{Y_0} = {x_0},} \\
			{Y_{{t_{n + 1}}}} = {Y_{{t_n}}} + \frac{{f({Y_{{t_n}}}){\delta _{mean}} + g({Y_{{t_n}}})(W((n + 1){\delta _{mean}}) - W(n{\delta _{mean}}))}}{{1 + {{({\delta _{mean}})}^{\gamma } }|f({Y_{{t_n}}})| + {{({\delta _{mean}})}^{\gamma} }|g({Y_{{t_n}}}){|^2}}} 
		\end{array}\right.
	\end{split}
\end{equation}
with $n = 0,1,...$, where both drift and diffusion coefficients $f, g$ may obey  the super-linear growth condition. The strong convergence of order $1/2$ is achieved by setting $\gamma = 1/2$.
\item[(3)] \textbf{Backward EM scheme} (see \cite{kloeden1992numerical}) The fixed-step backward EM scheme (\textbf{FS-BEM}) is given   by
\begin{equation}   \label{BEM}
	\begin{split}
		\left\{\begin{array}{l}
			Y_0 = x_0, \\
			{Y_{{t_{n + 1}}}} = {Y_{{t_n}}} + f({Y_{{t_{n + 1}}}}){\delta _{mean}} + g({Y_{{t_n}}})(W((n + 1){\delta _{mean}}) - W(n{\delta _{mean}})),
		\end{array}\right.
	\end{split}
\end{equation}
with $n= 0,1,...$, where $f(\cdot)$ satisfies the super-linear growth and $g(\cdot)$ satisfies the linear growth condition. The rate of strong convergence is proved to be $1/2$. 

\end{itemize}

\begin{example}{\rm
	Consider the stochastic Ginzburg--Landau equation, which describes a phase transition from the theory of superconductivity (see e.g., \cite{kloeden1992numerical, hutzenthaler2015numerical})
	\begin{equation}   \label{ex1}
		dX(t) = \left[ {\left( {{a_1} + \frac{1}{2}a_2^2} \right)X(t) - {X^3}(t)} \right]dt + {a_2}X(t)dW(t), \quad  X(0)= x_0 = 1.
	\end{equation}
	It admits the following explicit solution:
	\begin{equation} \label{ex1_1}
		\begin{split}
			X(t) = \frac{{{x_0}\exp ({a_1}t + {a_2}W(t))}}{{\sqrt {1 + 2x_0^2\int_0^t {\exp (2{a_1}s + 2{a_2}W(s))ds} } }}, \quad t \ge 0.
		\end{split}
	\end{equation}
	Set $ a_1 = -3/2$ and $a_2 = 1$. It can be verified that Assumptions \ref{A2.1} and \ref{A3.1} hold with $p,p^* > 2$ and $l = 2$. Then by virtue of Theorem \ref{rate1} we see, ATS-TEM admits a $1/2$-order of convergence rate.

	Next, we use the ATS-TEM \eqref{scheme}-\eqref{eq-tem} to carry out the numerical experiments. Set $\hat{\delta} = 2^{-20}$ and $\check{\delta} = 2^{-12}$.
	Figure \ref{fig:convergence} plots the convergence of ATS-TEM with different maximum stepsizes. It is evident from Figure \ref{fig:convergence} that ATS-TEM has a $1/2$-order of convergence rate, which is consistent with our theoretical result.
	Then, we compare the performance of ATS-TEM with FS-TEM, FS-TaEM, FS-BEM schemes that use the uniform timestep $\delta_{mean}$. 
	Figure \ref{2ex_4} shows the comparisons of RMSE against the number of timesteps and CPU time. For a given RMSE=0.005, Figure \ref{2ex_4} (Left) presents that ATS-TEM costs slightly fewer numbers of timesteps than FS-TEM, FS-TaEM and FS-BEM. And we observe from Figure \ref{2ex_4} (Right) that ATS-TEM costs less CPU time than FS-TEM and FS-TaEM, and substantially less CPU time than FS-BEM for a given RMSE= 0.005.
	Figure \ref{1_compare} (Left) plots the sample paths of the ATS-TEM and the exact solution \eqref{ex1_1} for SDE \eqref{ex1}. One observes that ATS-TEM effectively approximates the exact solution. With the uniform timestep $\delta_{mean} = 0.3704$, Figure \ref{1_compare} (Right) shows that the fixed-step EM scheme may lead to numerical explosion. These show that for SDE \eqref{ex1}, the ATS-TEM performs slightly better than fixed-step schemes that use the uniform timestep.
	
	\begin{figure}[H]
		\centering
		\includegraphics[width=0.65\textwidth]{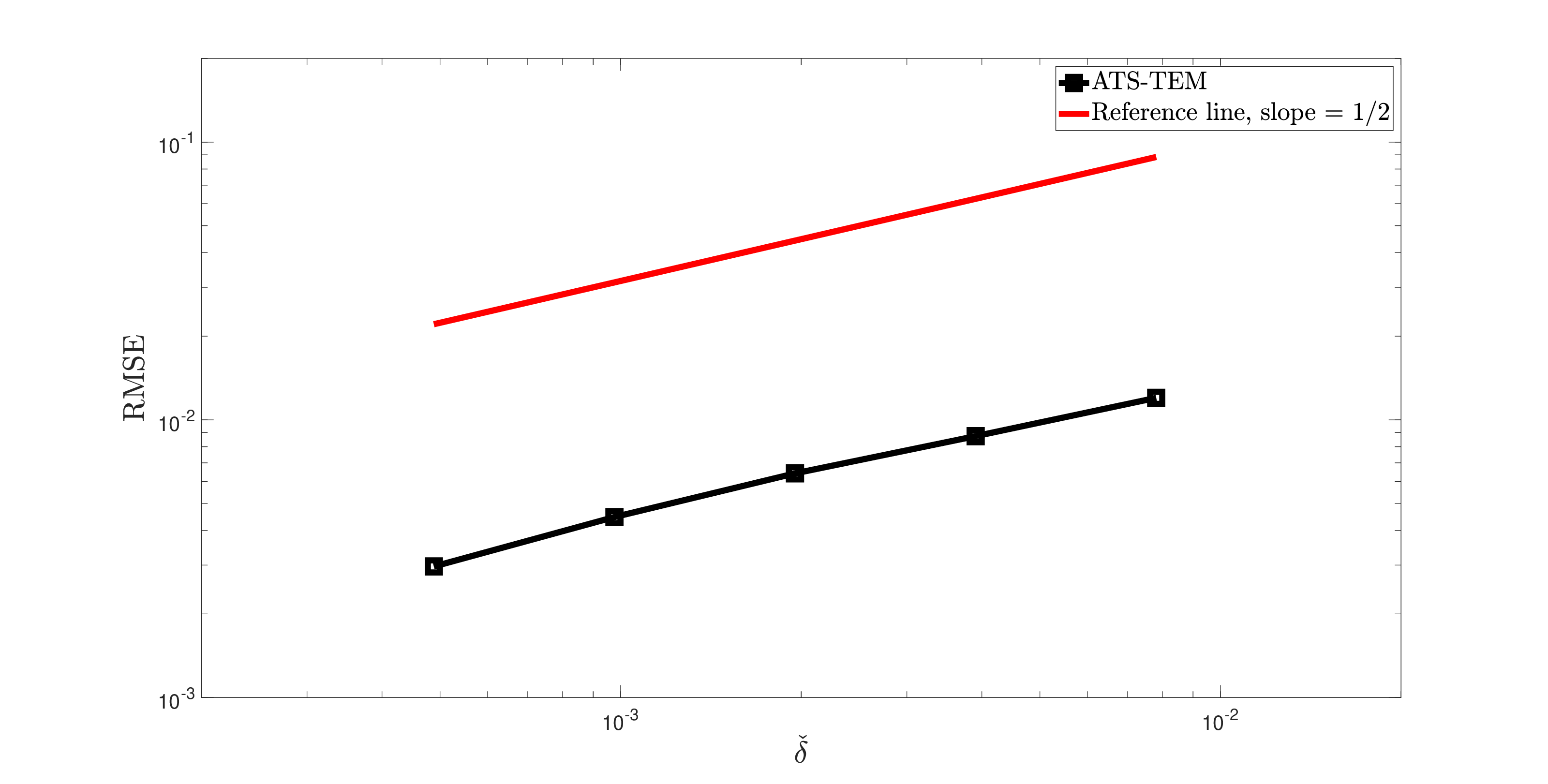}
		\caption{The strong convergence rate of ATS-TEM for SDE \eqref{ex1}.}
		\label{fig:convergence}
	\end{figure}
	
	\begin{figure}[htbp]
		\centering
		\begin{subfigure}{0.48\textwidth}
			\centering
			\includegraphics[width=\textwidth]{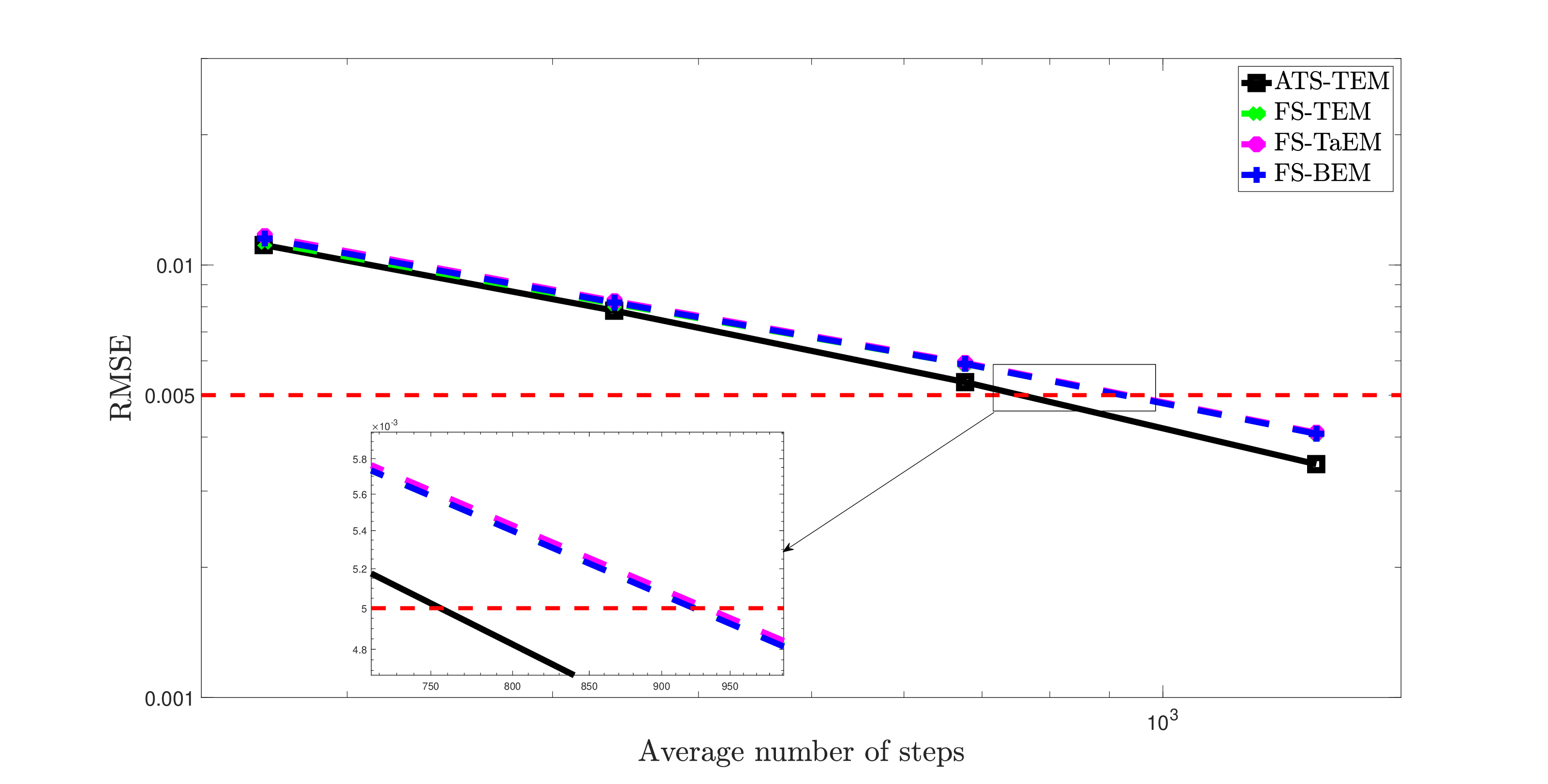}
		\end{subfigure}
		\hspace{0.01\textwidth}
		\begin{subfigure}{0.48\textwidth}
			\centering
			\includegraphics[width=\textwidth]{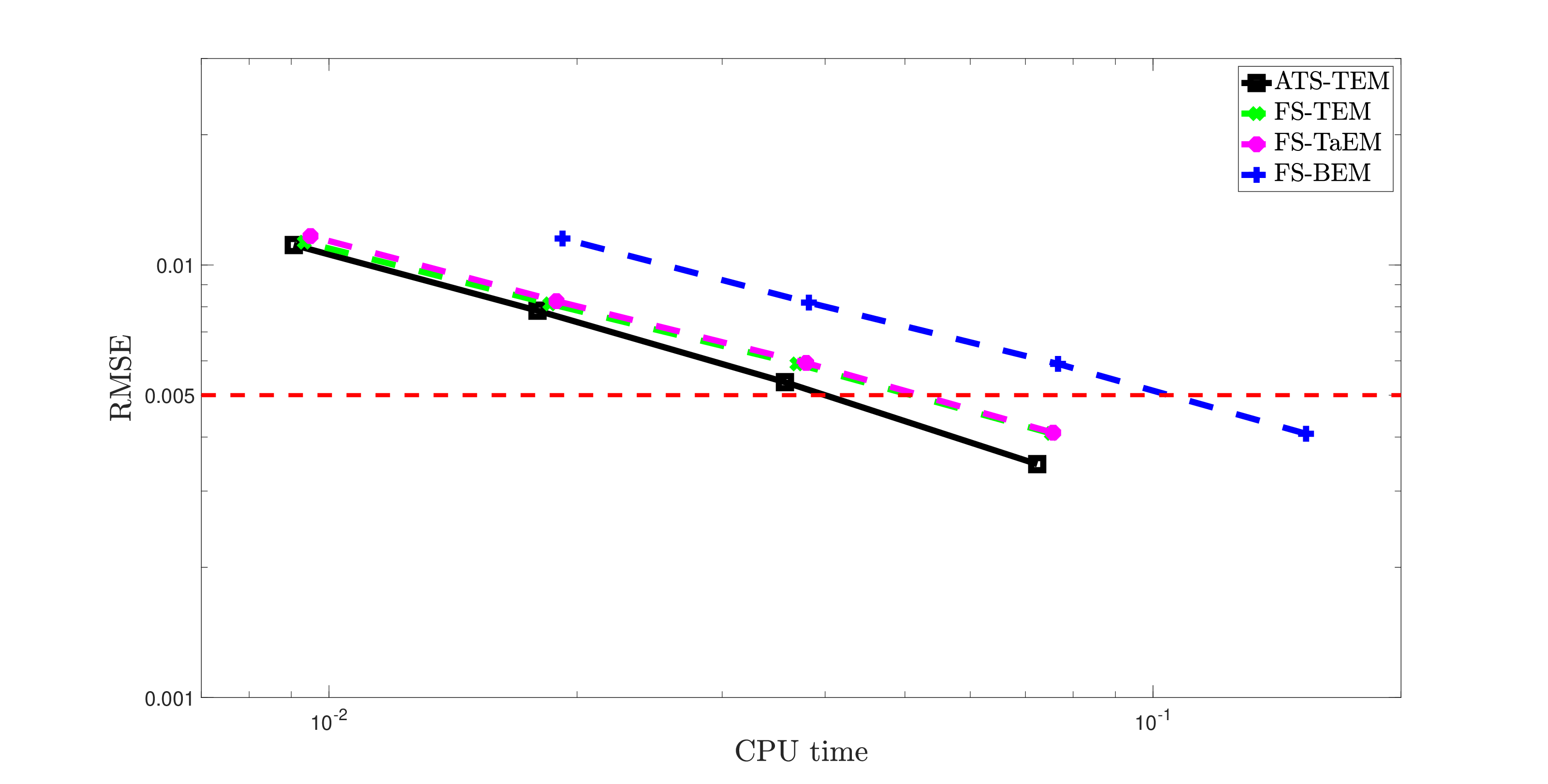}
		\end{subfigure}
		\caption{RMSE of numerical schemes for SDE \eqref{ex1}.}
		\label{2ex_4}
	\end{figure}
	
	\begin{figure}[H]
		\centering
		\begin{subfigure}{0.48\textwidth}
			\centering
			\includegraphics[width=\textwidth]{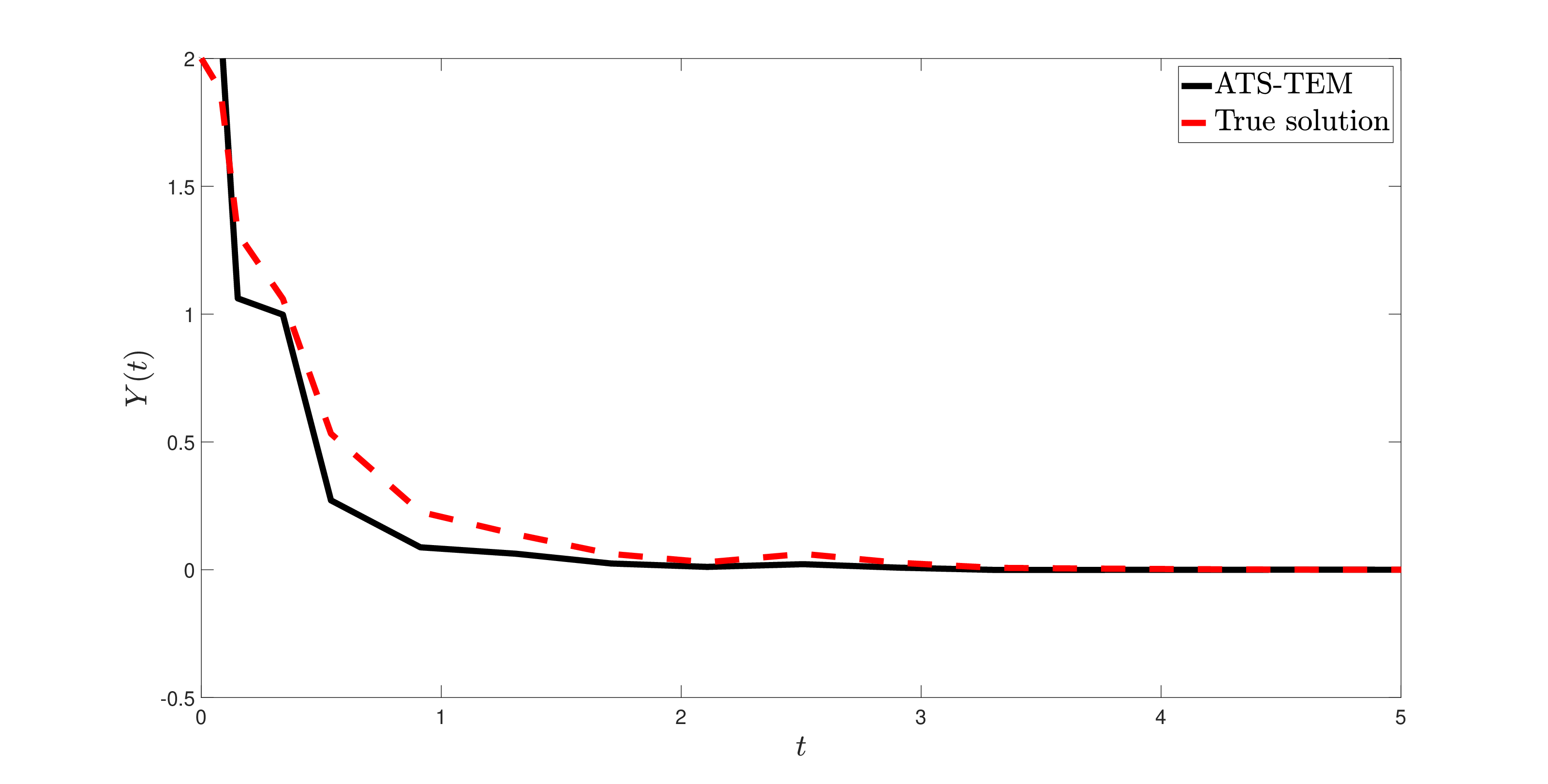}
		\end{subfigure}
		\hspace{0.01\textwidth}
		\begin{subfigure}{0.48\textwidth}
			\centering
			\includegraphics[width=\textwidth]{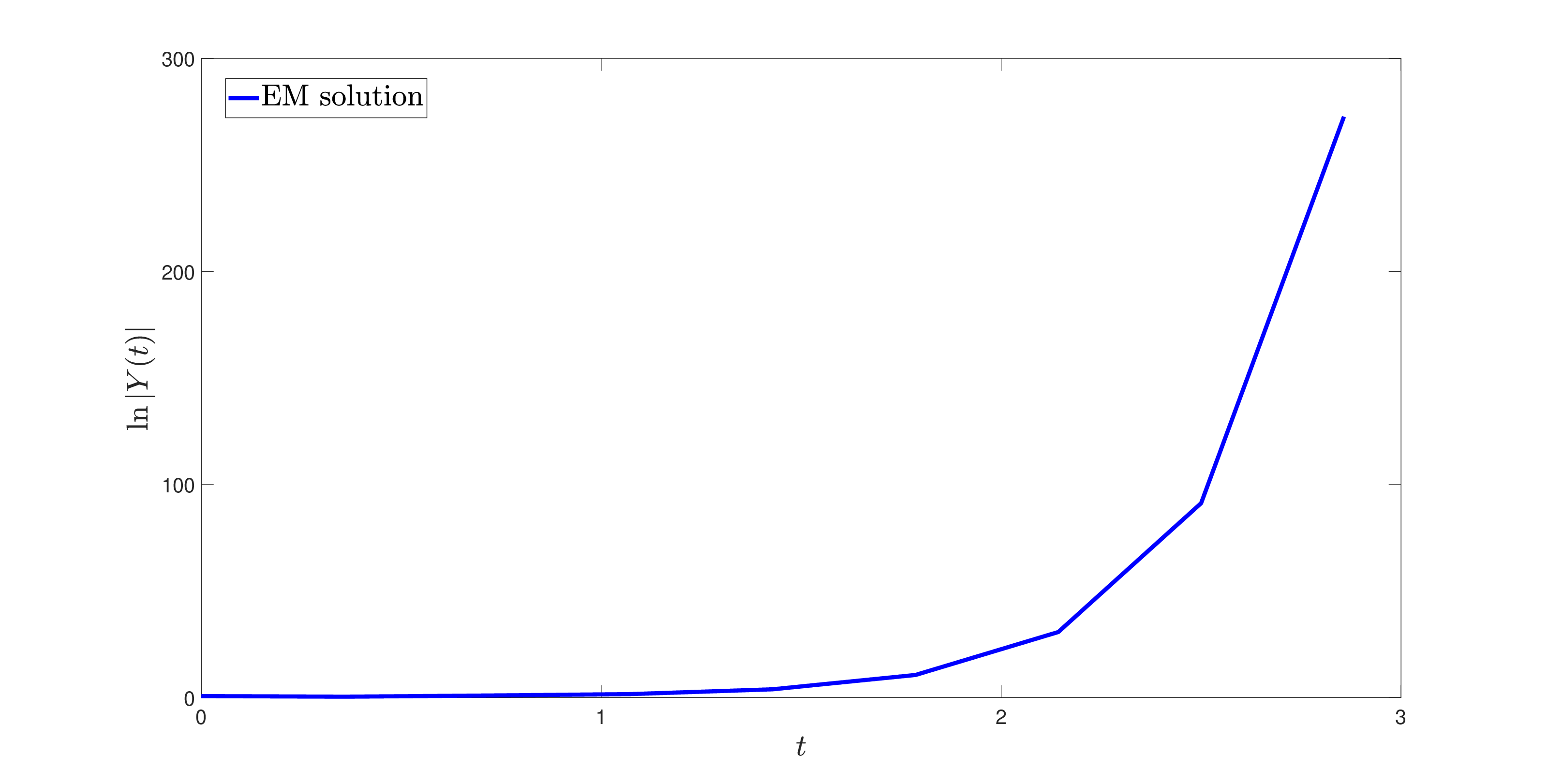}
		\end{subfigure}
		\caption{(Left) Sample paths of ATS-TEM solution and exact solution \eqref{ex1_1} for SDE \eqref{ex1};
			(Right) Sample path of the fixed-step EM solution $\ln|Y(t)|$ with the uniform timestep $\delta_{mean}$.}
		\label{1_compare}
	\end{figure}
	}
\end{example}

\begin{example} 
 {\rm  Consider a scalar stiff SDE with a double-well drift potential (see. e.g., \cite{rackauckas2020stability}) 
\begin{equation} \label{ex0}
	dX(t) =  - V'(X(t))dt + 10X(t)dW(t), \quad X(0) = X_0 = 3,
\end{equation}
where 
\begin{equation*}
	V(x) = \frac{{{x^4}}}{4} - \frac{{26}}{3}{x^3} + \frac{{125}}{2}{x^2} - 100x, \quad -V'(x) = (x - 1)(5 - x)(x - 20).
\end{equation*}
A direct computation implies that Assumptions \ref{A2.1} and \ref{A3.1} hold with $p, p^* >2$ and $l = 2$. It follows from Theorem \ref{th-5.3} that ATS-TEM \eqref{scheme}-\eqref{eq-tem} has a $1/2$-order of convergence rate.

Next, we are going to compare the performance of ATS-TEM with schemes that use a uniform timestep $\delta_{mean}$. 
Set $\hat{\delta} = 2^{-24}$ and $\check{\delta} = 2^{-12}$. 
Since the exact solution of SDE \eqref{ex0} cannot be given explicitly, we take the solution generated by the same method with a timestep that is four times smaller as the reference solution. 
Figure \ref{fig:d1} displays a sample path generated by the ATS-TEM for SDE \eqref{ex0}, together with the corresponding adaptive timesteps at each step.
One observes that the timestep is reduced in regions where the path exhibits rapid oscillations, while larger timesteps are used in relatively smooth regions.
Figure \ref{fig:d5_23} presents the comparisons of ATS-TEM and FS-TaEM schemes, where the left one is the RMSE against the average number of timesteps, and the right one is about the RMSE against the CPU time. 
For a given RMSE=0.2, one observes from Figure \ref{fig:d5_23} (Left) that the ATS-TEM costs fewer numbers of timesteps than FS-TaEM with the uniform timestep, and in Figure \ref{fig:d5_23} (Right), the ATS-TEM costs less CPU time than FS-TaEM.

\begin{figure}[h]
	\centering
	\includegraphics[width=0.65\textwidth]{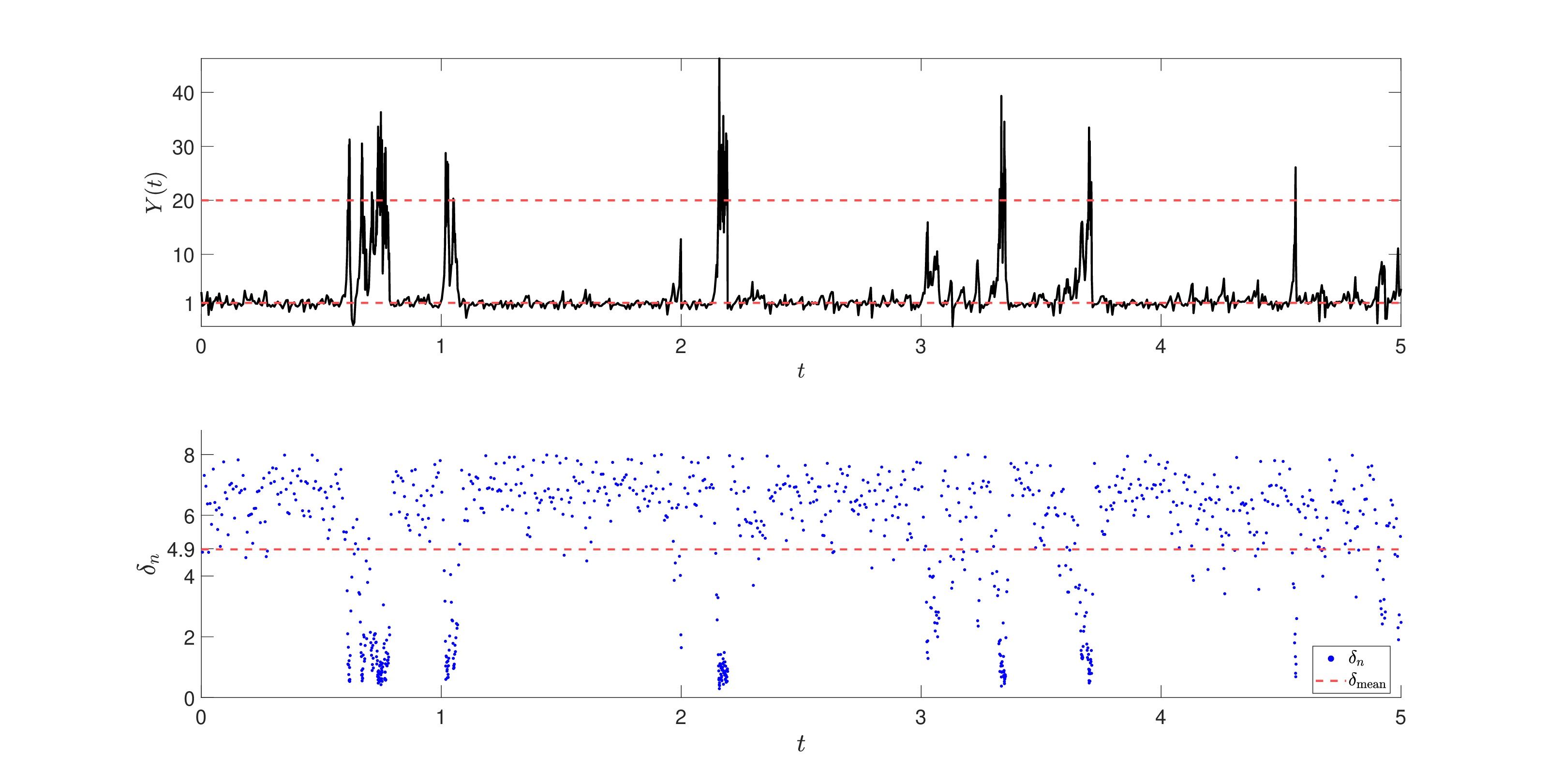}
	\caption{One sample path generated by ATS-TEM and the corresponding adaptive timesteps for SDE \eqref{ex0}.}
	\label{fig:d1}
\end{figure}

\begin{figure}[h]
	\centering
	\begin{subfigure}{0.48\textwidth}
		\centering
		\includegraphics[width=\textwidth]{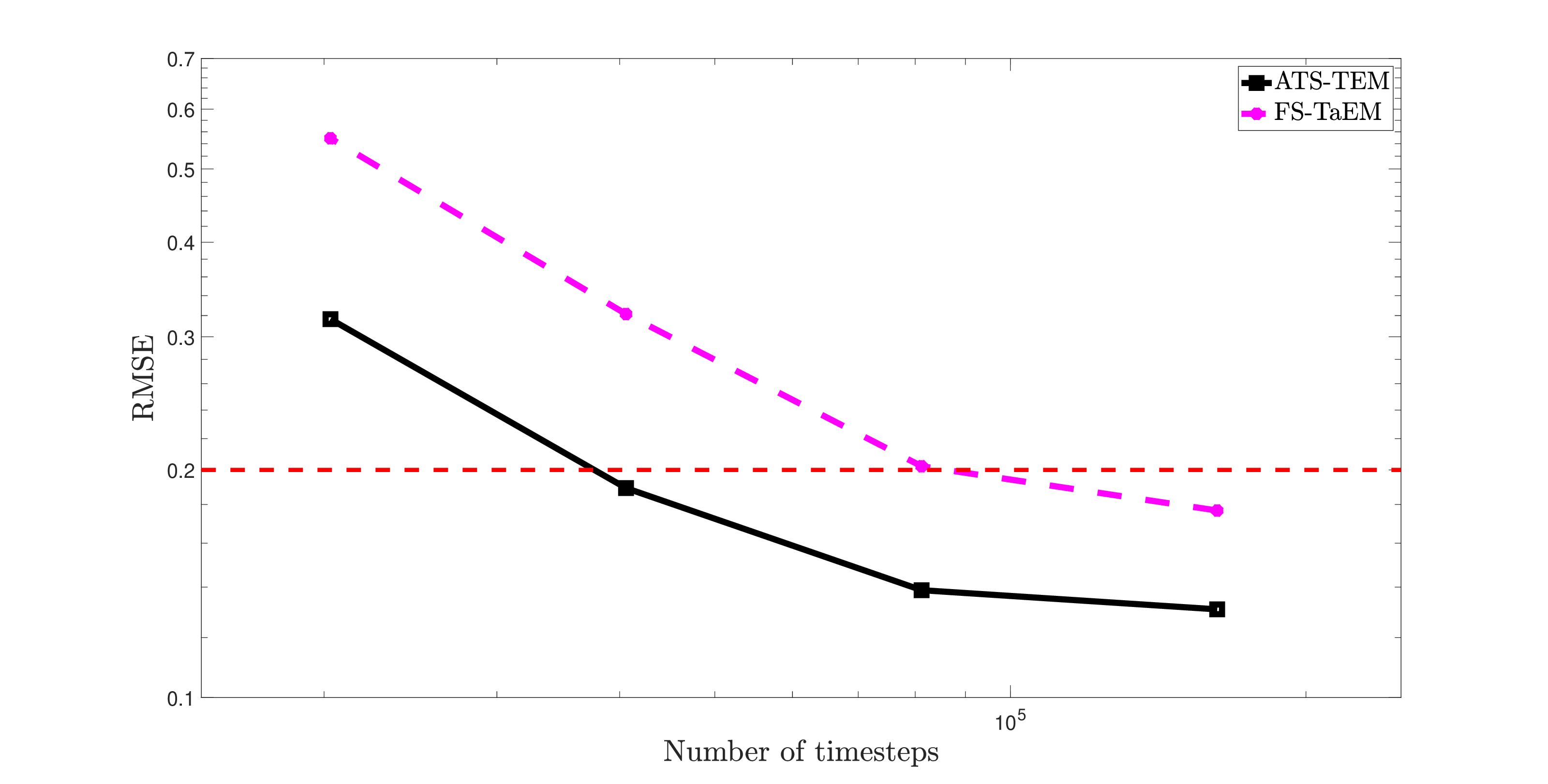}
	\end{subfigure}
	\hspace{0.01\textwidth}
	\begin{subfigure}{0.48\textwidth}
		\centering
		\includegraphics[width=\textwidth]{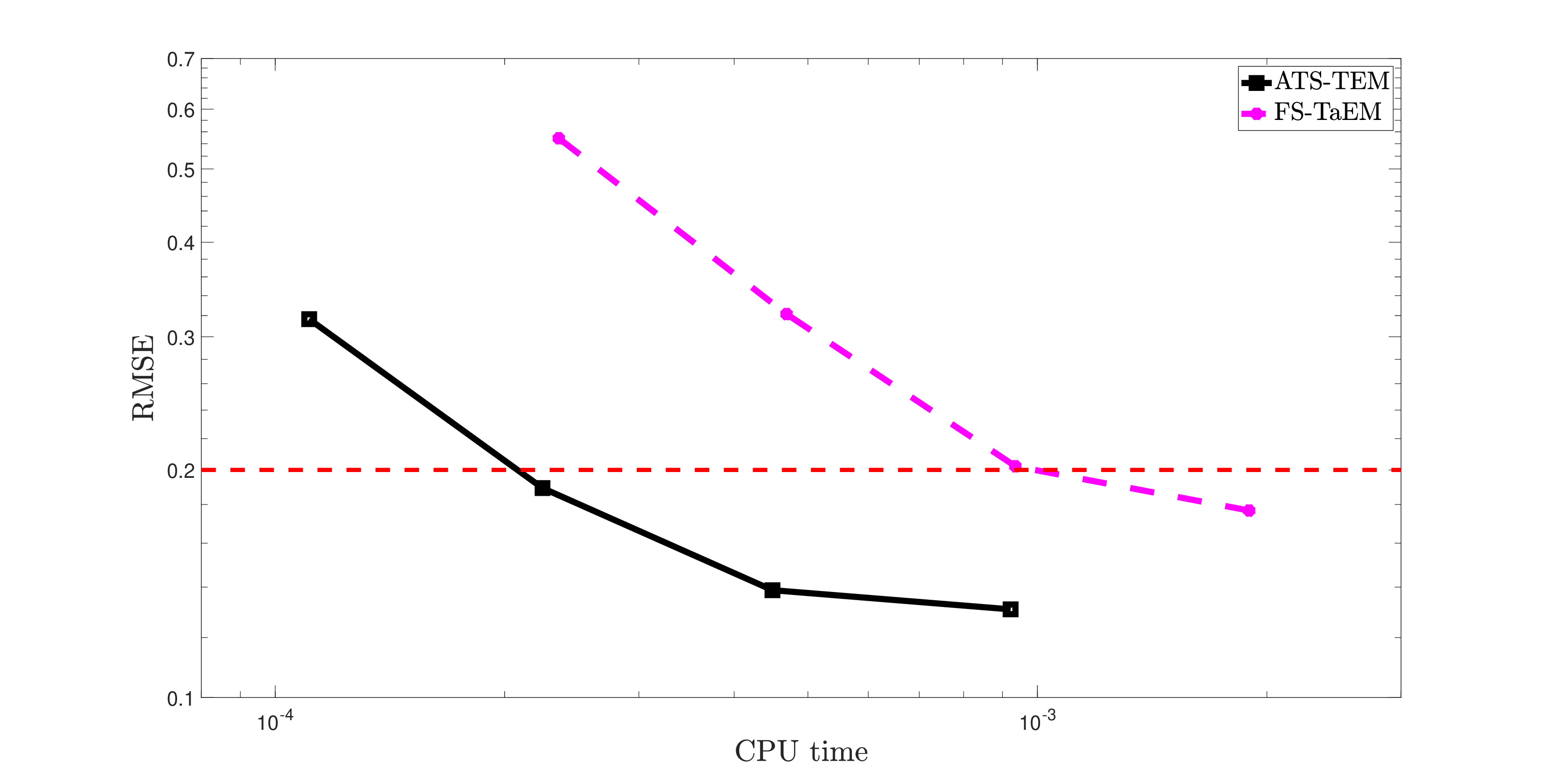}
	\end{subfigure}
	\caption{RMSE of ATS-TEM and FS-TaEM for SDE \eqref{ex0}.} 
	\label{fig:d5_23}
\end{figure}

We further compare ATS-TEM with the adaptive EM scheme (denoted by AEM below) proposed by Fang and Giles \cite{fang2020adaptive}, 
\begin{equation}   \label{AEM}
	\begin{split}
		\left\{\begin{array}{l}
			Z_0 = x_0, \quad t_0 = 0, \\
			\delta _n^h:= \frac{{\max (1,|Z_{t_n}|)}}{{\max (1,|f(Z_{t_n})|)}}h, \quad t_{n+1} = t_n + \delta_n^h, \quad n = 0,1,..., \\
			{Z_{{t_{n + 1}}}} = {Z_{{t_n}}} + f({Z_{{t_{n }}}})\delta _n^h + g({Z_{{t_n}}})(W({t_{n + 1}}) - W({t_n})), \quad n = 0,1,...
		\end{array}\right.
	\end{split}
\end{equation}
with $0 \le h \le 1$.
Figure \ref{fig:d6_12} shows the comparisons of ATS-TEM and AEM schemes for RMSE against the number of timesteps and CPU time. For a given RMSE=0.2, it can be observed that in Figure \ref{fig:d6_12} (Left), the ATS-TEM costs fewer number of timesteps than AEM, and that in Figure \ref{fig:d6_12} (Right), the ATS-TEM costs less CPU time than AEM. 
These show the better performance of ATS-TEM for the stiff SDE \eqref{ex0}.

\begin{figure}[H]  
	\centering
	\begin{subfigure}{0.48\textwidth}
		\centering
		\includegraphics[width=\textwidth]{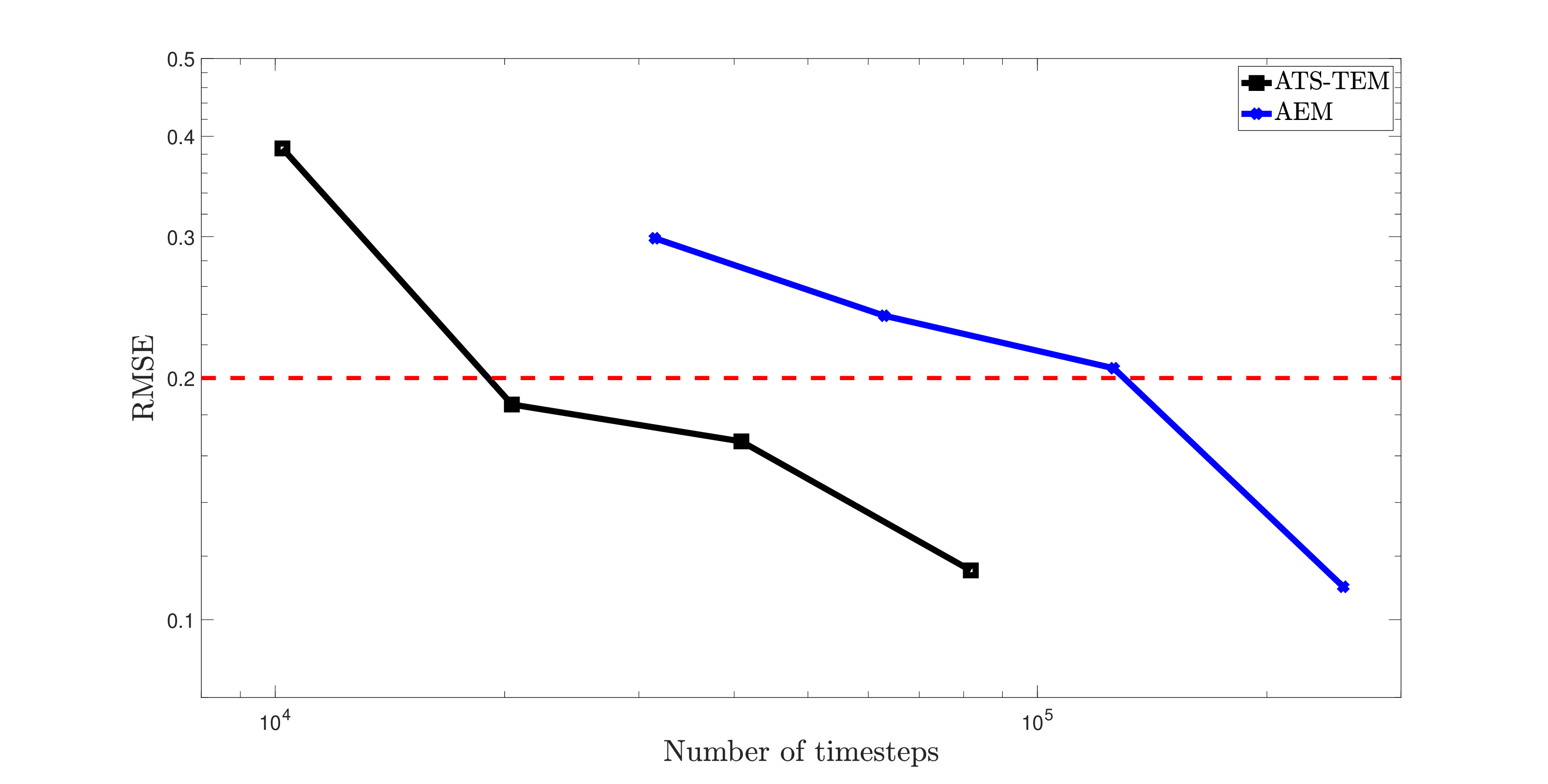}
	\end{subfigure}
	\hspace{0.01\textwidth}
	\begin{subfigure}{0.48\textwidth}
		\centering
		\includegraphics[width=\textwidth]{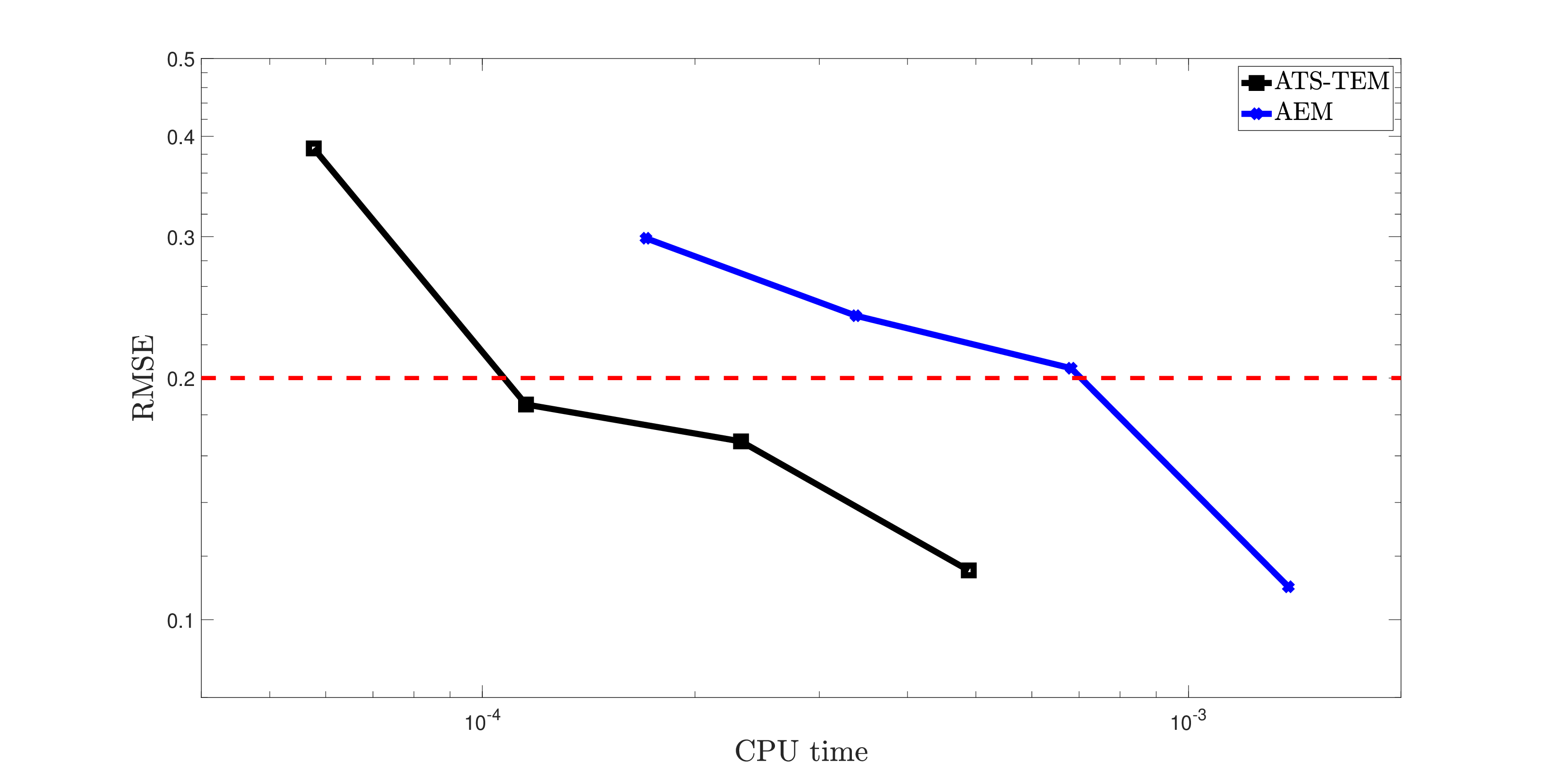}
	\end{subfigure}
	\caption{RMSE of ATS-TEM and AEM for SDE \eqref{ex0}.} 
	\label{fig:d6_12}
\end{figure}
}
\end{example}

\begin{example}  \label{ex3}{\rm 
The Heston 3/2 model is a stochastic volatility model widely used in mathematical finance. In this model, the variance process is described by 
\begin{equation}  \label{heston}
	dX(t) = \kappa X(t)(\theta  - X(t))dt + \sigma {X^{\frac{3}{2}}}(t)dW(t), \quad X(0)=x_0 \in \mathbb{R}_+,
\end{equation}
where $\kappa >0$ represents the mean-reversion speed, determining how quickly the volatility returns to its long-term equilibrium; $\theta  >0$ denotes the long-term average volatility level; and $\sigma  >0$ is the volatility of volatility, which quantifies the magnitude of random fluctuations in the volatility process itself (see e.g., \cite{lewis2016option, wu2025an}). Let $(\kappa, \theta, \sigma, x_0) = (2, 0.04, 0.5, 1)$, one can verify that Assumption \ref{A2.1} holds. It follows from Theorem \ref{strong-con2} that ATS-TEM \eqref{scheme}-\eqref{eq-tem} converges strongly to the solution of SDE \eqref{heston} in $L^q$ sense for $q > 2$.

Next, we use the ATS-TEM \eqref{scheme}-\eqref{eq-tem} to carry out the numerical experiments.
Set $\hat{\delta} = 2^{-18}$.
Since the exact solution of SDE \eqref{heston} cannot be given explicitly, we take the ATS-TEM solution with a sufficiently small maximal timestep $\check{\delta} = 2^{-12}$ as the reference solution.
The $L^4$ strong error of ATS-TEM for SDE \eqref{heston} with different maximum timesteps is listed in Table \ref{table1}, it is evident that the $L^4$ strong error decreases as the maximal timestep $\check{\delta}$ decreases.

\begin{table}[htbp] 
	\centering
	\caption{$L^4$ strong error of the ATS-TEM for SDE \eqref{heston} with different maximum stepsizes}  
	\label{table1}
	\begin{tabularx}{0.8\textwidth}{>{\centering\arraybackslash}X>{\centering\arraybackslash}X}  
		\toprule  % 使用booktabs的顶部线
		$\check{\delta}$ & $\sqrt[4]{{\mathbb{E}{{\left| {X(T) - Y(T)} \right|}^4}}}$ \\
		\midrule  % 使用booktabs的中部线
		$2^{-6}$  & $3.6541\times 10^{-3}$ \\
		$2^{-7}$  & $2.2019\times 10^{-3}$ \\
		$2^{-8}$  & $1.5744\times 10^{-3}$ \\
		$2^{-9}$  & $1.1021\times 10^{-3}$ \\
		$2^{-10}$ & $7.2190\times 10^{-4}$ \\
		\bottomrule  % 使用booktabs的底部线
	\end{tabularx}
\end{table}}
\end{example}

\begin{example}{\rm 
The stochastic Lorenz system is a simplified model for atmospheric convection rolls and is well known for its chaotic dynamics,
\begin{equation}   \label{lorenz}  
	\begin{split}
		\left\{\begin{array}{l}
			{d{X_1}(t) ={\alpha _1} \left( {{X_2}(t) - {X_1}(t)} \right)dt + {\beta _1}{X_2}(t)d{W_1}(t)}, \\
			{d{X_2}(t) = \left( {{\alpha _2}{X_1}(t) - {X_2}(t) - {X_1}(t){X_3}(t)} \right)dt + {\beta _2}{X_2}(t)d{W_2}(t)},  \quad X(0) = x_0 \in \mathbb{R}^3.\\
			{d{X_3}(t) = \left( {{X_1}(t){X_2}(t) - {\alpha _3}{X_3}(t)} \right)dt + {\beta _3}{X_3}(t)d{W_3}(t)}.
		\end{array}\right.	\end{split}
\end{equation}

With the parameters $({\alpha _1},{\alpha _2},{\alpha _3}) = (10,28,8/3)$, the system exhibits its characteristic chaotic behavior, evolving toward the butterfly-shaped Lorenz attractor \cite{hutzenthaler2015numerical, lorenz1963deterministic}. Set $\beta_i = 0.5$ for $i = 1,2,3$, and initial value $x_0 = (1, 1, 20)^{\rm T}$. It can be verified that Assumption \ref{A2.1} holds with $p > 2$. Then, by virtue of Theorem \ref{strong-con2}, ATS-TEM solution converges strongly to the solution of SDE \eqref{lorenz} in $L^q$ sense for $q > 2$. Figure \ref{l1} (Left) displays a sample path generated by ATS-TEM, and Figure \ref{l1} (Right) plots its corresponding adaptive timesteps at each step. 
One observes that more than half of the timesteps are larger than the uniform timestep $\delta_{mean} = 2.8 \times 10^{-4}$.

Next, we are going to compare the performance of ATS-TEM with FS-TEM, FS-TaEM, FS-BEM schemes that use a uniform timestep. Set $\hat{\delta} = 2^{-20}$ and $\check{\delta} = 2^{-12}$.  
Since the exact solution of SDE \eqref{lorenz} cannot be given explicitly, we take the solution generated by the same method with a timestep that is four times smaller as the reference solution. 
Figure \ref{g3_12} shows the comparisons of RMSE against the number of timesteps and CPU time. For a given RMSE=0.5, we observe from Figure \ref{g3_12} (Left) that ATS-TEM costs slightly fewer timesteps than FS-TEM and FS-BEM, and substantially fewer number of timesteps than FS-TaEM. And we observe from Figure \ref{g3_12} (Right) that ATS-TEM costs less CPU time than FS-TEM, FS-TaEM and FS-BEM for a given RMSE = 0.5. 
These show the better performance of ATS-TEM for the stochastic Lorenz system \eqref{lorenz}.

\begin{figure}[htbp]
	\centering
	\begin{subfigure}{0.48\textwidth}
		\centering
		\includegraphics[width=\textwidth]{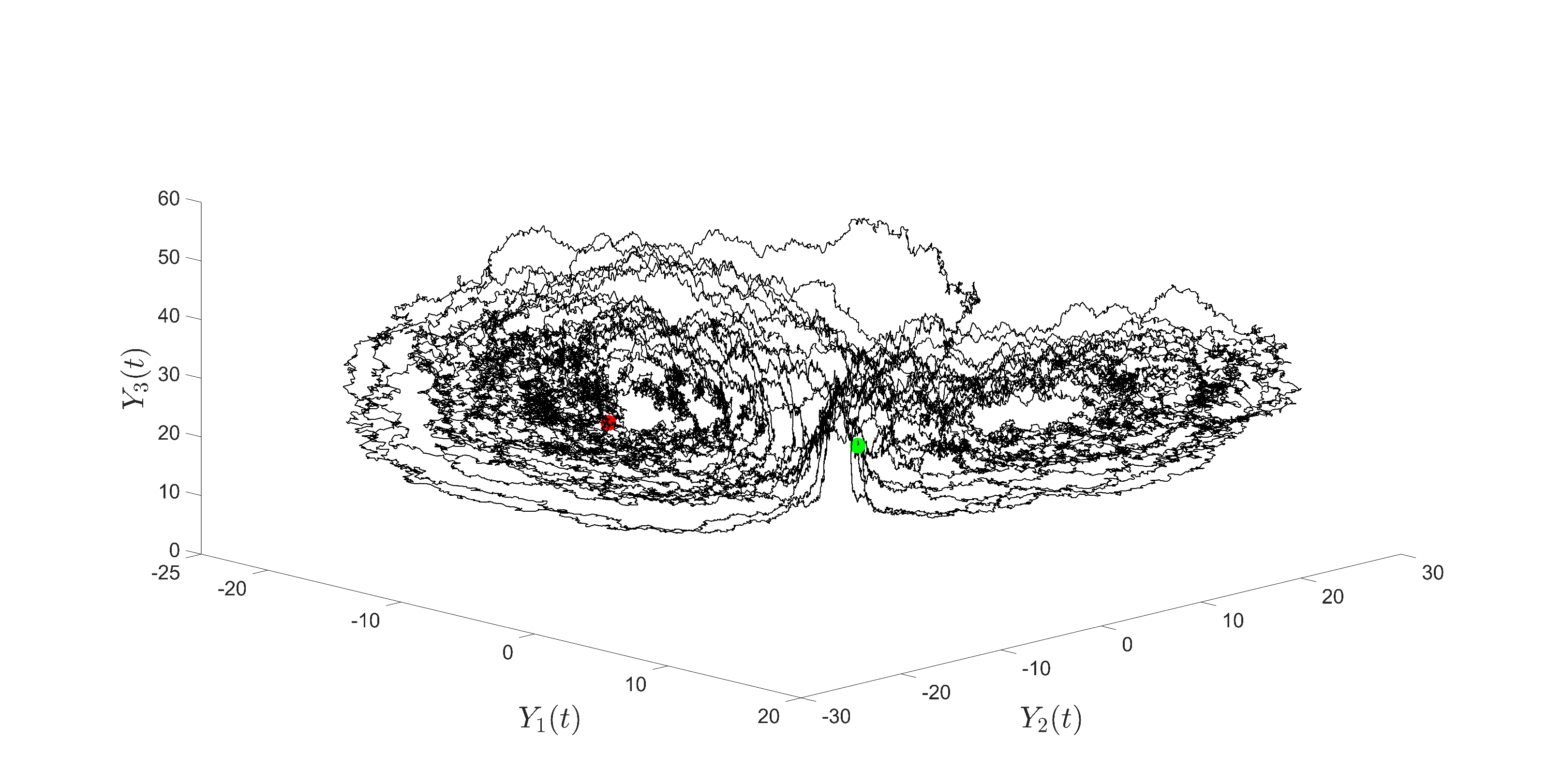}
	\end{subfigure}
	\hspace{0.01\textwidth}
	\begin{subfigure}{0.48\textwidth}
		\centering
		\includegraphics[width=\textwidth]{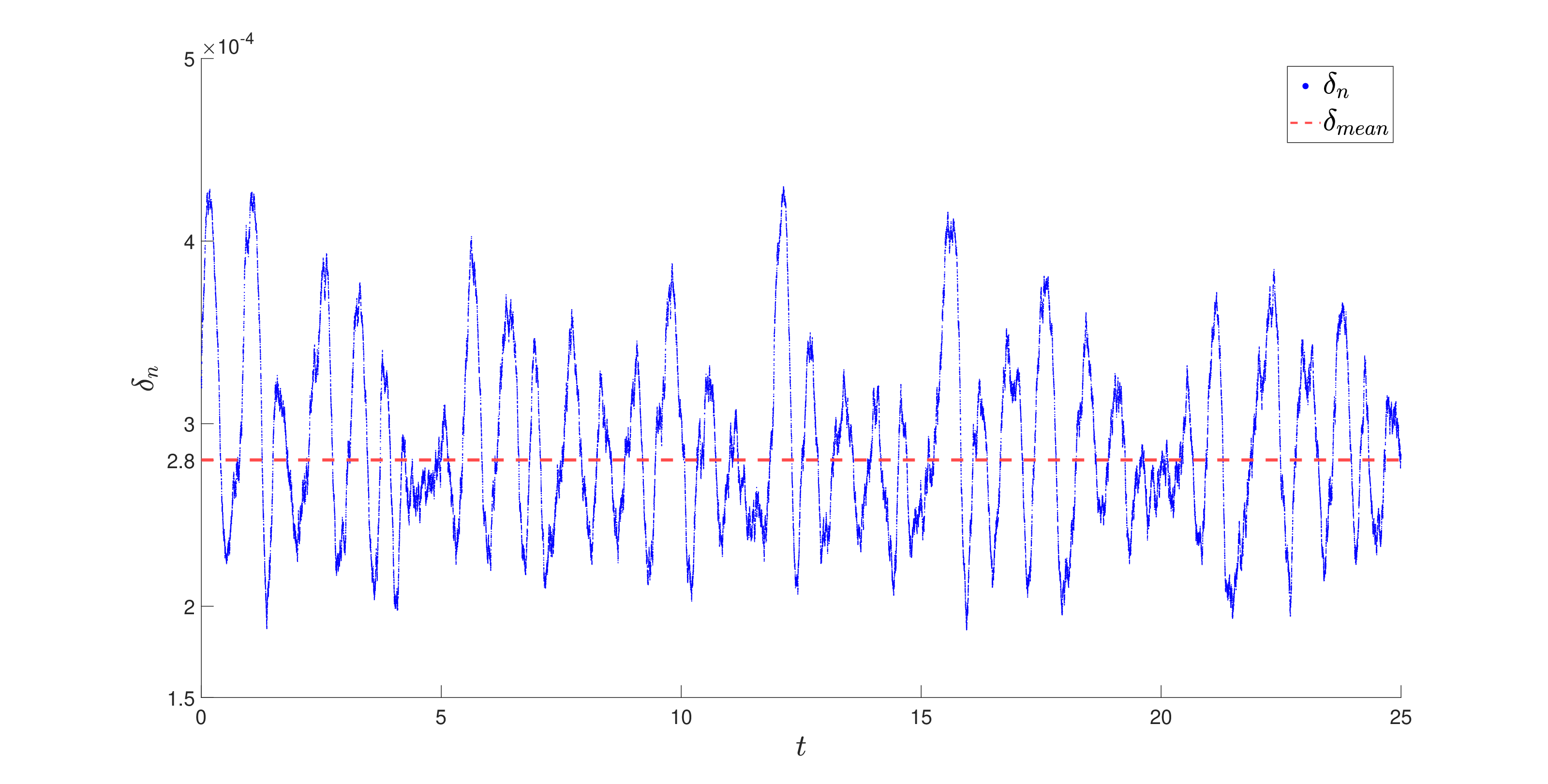}
	\end{subfigure}
	\caption{(Left) One sample path generated by ATS-TEM for SDE \eqref{lorenz}, the green marker indicates the initial state of the process, while the red marker denotes its terminal state at the final time; (Right) The corresponding adaptive timesteps generated along the same path.}
	\label{l1} 
\end{figure}

\begin{figure}[htbp]
	\centering
	\begin{subfigure}{0.48\textwidth}
		\centering
		\includegraphics[width=\textwidth]{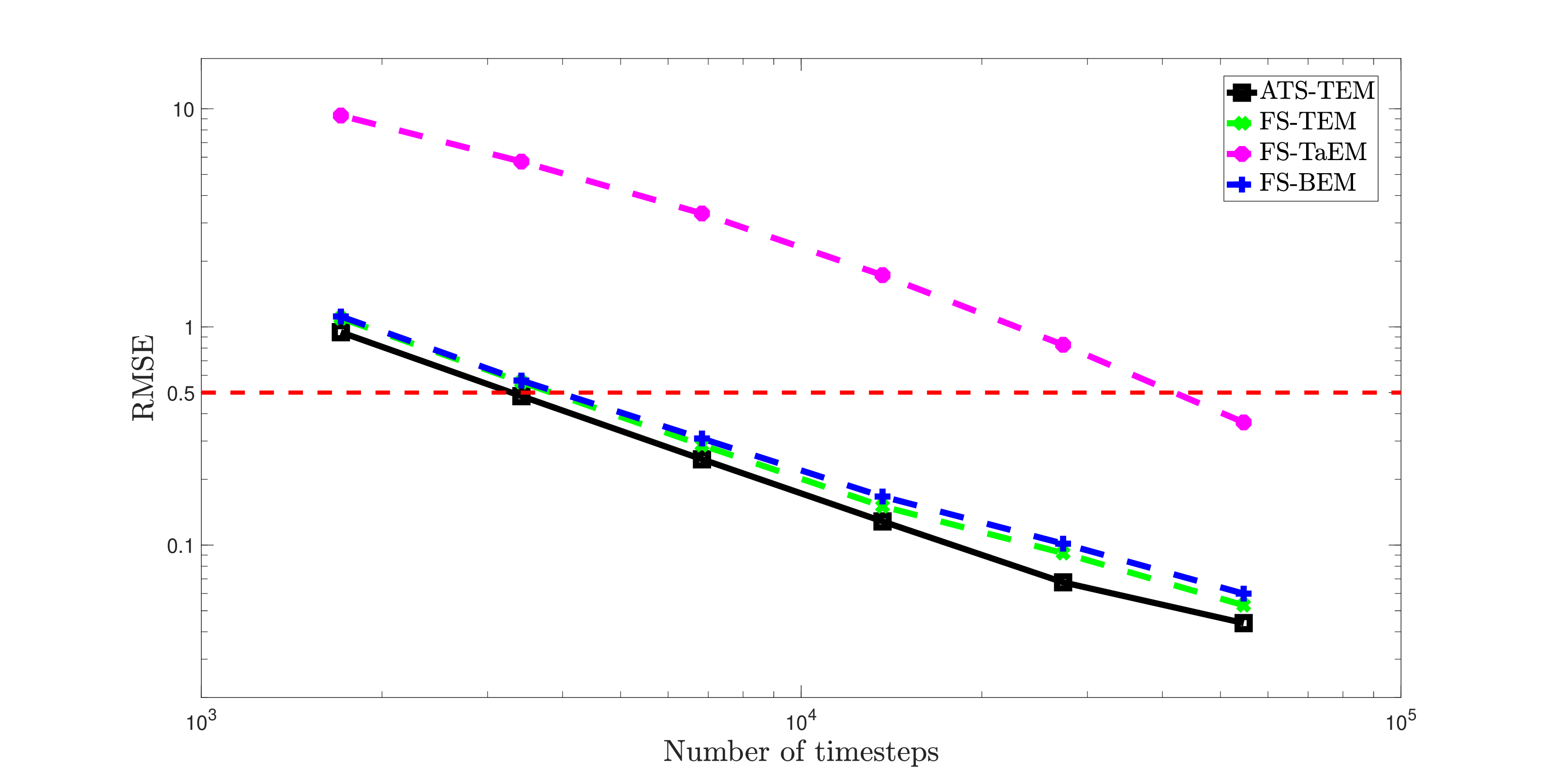}
	\end{subfigure}
	\hspace{0.01\textwidth}
	\begin{subfigure}{0.48\textwidth}
		\centering
		\includegraphics[width=\textwidth]{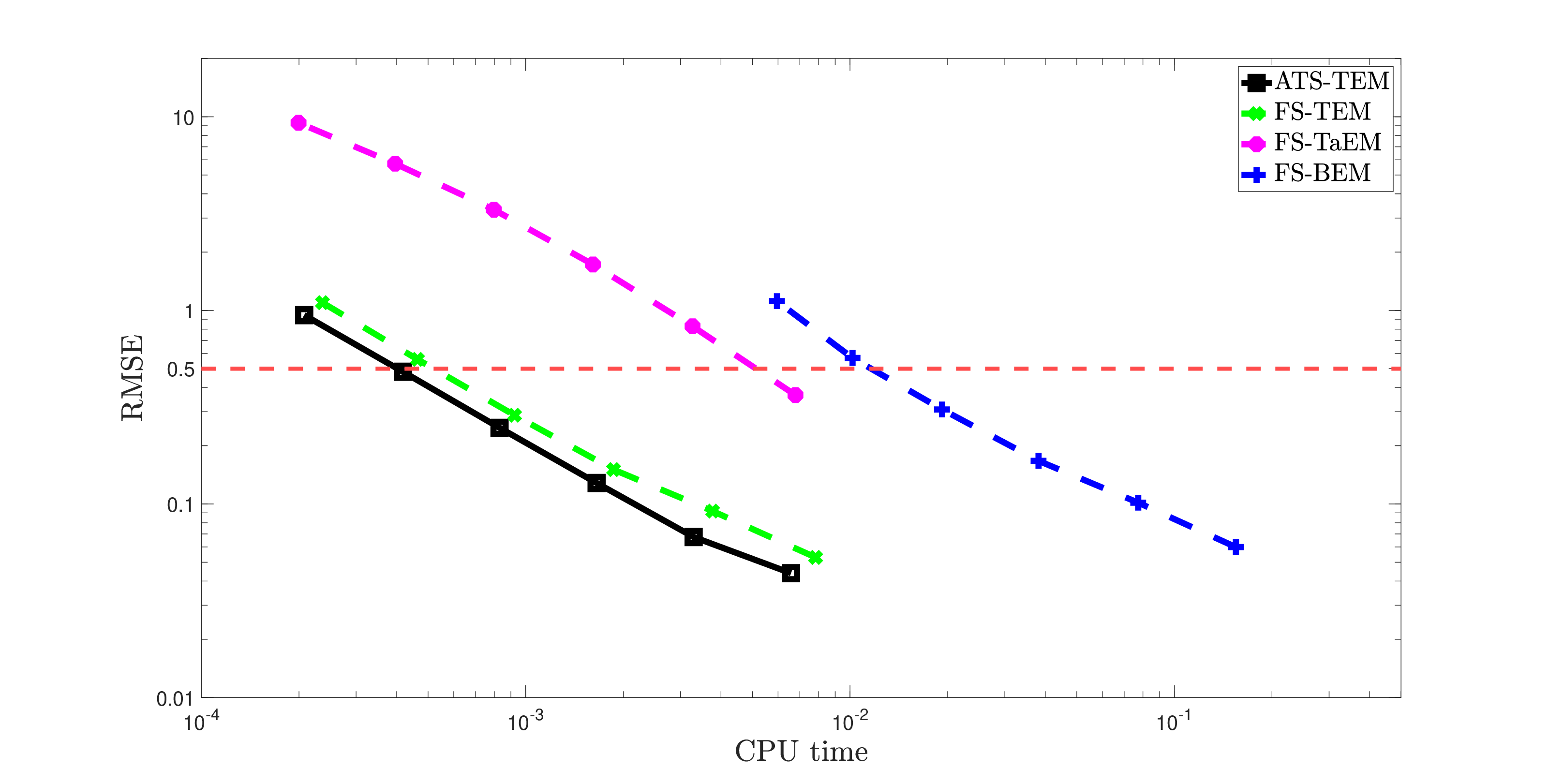}
	\end{subfigure}
	\caption{RMSE of numerical schemes for SDE \eqref{lorenz}.}
	\label{g3_12} 
\end{figure}
}
\end{example}

\section{Conclusions}  \label{sec7}
This paper proposes an ATS-TEM scheme for SDEs whose drift and diffusion coefficients are both have polynomial growth. We establish the $p$th moment boundedness of the numerical solutions and prove the strong convergence of the proposed scheme. Under somewhat stronger conditions, we further obtain a $1/2$-order of strong convergence rate in $L^q$-sense ($q > 2$). Moreover, we show that FS-TaEM introduced in \cite{sabanis2016euler} can be employed as a backstop scheme, and the corresponding ATS-TaEM also achieves a $1/2$-order of convergence rate. Several numerical examples, including stiff and chaotic SDEs, are presented to validate the theoretical results and further verify the performance of the proposed scheme. In our future work, the preservation of the long-time asymptotic properties of the adaptive time-stepping method, including the uniform-in-time convergence rate and the ergodicity, will be reported.

\vskip 25pt
\bibliographystyle{elsarticle-num} 
\bibliography{ref}

\end{document}